\documentclass[journal]{IEEEtran}
\usepackage{blindtext}
\usepackage{graphicx}
\usepackage[utf8]{inputenc}
\usepackage{graphics} % for pdf, bitmapped graphics files
\usepackage{epsfig} % for postscript graphics files
\usepackage{mathptmx} % assumes new font selection scheme installed
\usepackage{amsmath} % assumes amsmath package installed
\usepackage{amssymb}  % assumes amsmath package installed
\usepackage{cite}
\usepackage{multicol}
\usepackage{mathtools, cuted}
\usepackage[cmintegrals]{newtxmath}
\usepackage{epstopdf}
\usepackage{graphics} % for pdf, bitmapped graphics files
\usepackage{epsfig} 

\usepackage{mathptmx} % assumes new font selection scheme installed
\usepackage{amsmath} % assumes amsmath package installed
\usepackage{amssymb}  % assumes amsmath package installed
\usepackage{cite}
\usepackage{multicol}
\usepackage{mathtools, cuted}
\usepackage[cmintegrals]{newtxmath}
\usepackage{mathtools}

\usepackage{mathptmx}
\usepackage{helvet}
\usepackage{courier}
\usepackage{subfigure}
\usepackage{type1cm}
\usepackage{titlesec}
\usepackage{epsfig} % for postscript graphics files
\usepackage{mathptmx} % assumes new font selection scheme installed
\usepackage{amsmath} % assumes amsmath package installed
\usepackage{amssymb}  % assumes amsmath package installed
\usepackage{cite}
\usepackage{multicol}
\usepackage{mathtools, cuted}
\usepackage[cmintegrals]{newtxmath}
\usepackage{makeidx}         % allows index generation
\usepackage{algorithm}% http://ctan.org/pkg/algorithm
\usepackage{algpseudocode}% http://ctan.org/pkg/algorithmicx
\DeclareMathOperator*{\argmin}{\arg\!\min}

\usepackage{multicol}        % used for the two-column index
\usepackage[bottom]{footmisc}% places footnotes at page bottom
\usepackage{caption}
\usepackage{nomencl}
\usepackage[usenames, dvipsnames]{color}

\usepackage{multicol}        % used for the two-column index
\usepackage[bottom]{footmisc}% places footnotes at page bottom
\usepackage{caption}
\usepackage{titlesec}
\usepackage{nomencl}
\usepackage[usenames, dvipsnames]{color}
\usepackage{multirow}

\usepackage{amsthm}
 
\theoremstyle{definition}
\newtheorem{definition}{Definition}

\theoremstyle{theorem}
\newtheorem{theorem}{Theorem}

\theoremstyle{remark}
\newtheorem{remark}{Remark}
\theoremstyle{proposition}
\newtheorem{proposition}{Proposition}
\theoremstyle{corollary}

\theoremstyle{proof}

\newtheorem{assumption}{Assumption}
\theoremstyle{assumption}

\theoremstyle{lemma}

\ifCLASSINFOpdf
\else
\fi
\begin{document}
\title{Integration of A* Search and Classic Optimal Control for Safe Planning of Continuum Deformation of a Multi-Quadcopter System }

\author{Hossein Rastgoftar% <-this % stops a space
\thanks{{\color{black}H. Rastgoftar is with the Department
of Aerospace Engineering, University of Michigan, Ann Arbor,
MI, 48109 USA e-mail: hosseinr@umich.edu.}}% <-this % stops a space
% \thanks{J. Doe and J. Doe are with Anonymous University.}% <-this % stops a space
% \thanks{Manuscript received April 19, 2005; revised August 26, 2015.}}
}

\markboth{
% IEEE Transactions on Control Systems Technology
}%
{Shell \MakeLowercase{\textit{et al.}}: Bare Demo of IEEEtran.cls for IEEE Journals}



\maketitle
\begin{abstract}
This paper offers an algorithmic approach to plan continuum deformation of a multi-quadcopter system (MQS) in an obstacle-laden environment. We treat the MQS as finite number of particles of a deformable body coordinating under a homogeneous transformation. In this context, we define the MQS homogeneous deformation coordination as a decentralized leader-follower problem, and integrate the principles of continuum mechanics, A* search method, and optimal control to safety and optimally plan MQS continuum deformation coordination. In particular, we apply the principles of continuum mechanics to obtain the safety constraints, use the A* search method to assign the intermediate configurations of the leaders by minimizing the travel distance of the MQS, and determine the leaders' optimal trajectories by solving a constrained optimal control problem.
% apply  obtain safety conditions for a deformable swarm coordination in a geometrically-constrained environment by eigen-decomposition of the homogeneous transformation, and incoroporate them as equality and inequality constraints imposed on planning of the leaders' desired trajectories. Given the safe A* search and optimal control methods to plan a safe continuum deformation coordination through optimize the leaders' desired trajectories via minimization of the travel distance between the initial and target locations of the MQS.
The optimal planning of the continuum deformation coordination is acquired by the quadcopter team in a decentralized fashion through local communication. 
% This paper studies the problem of leader-follower affine transformation of a large-scale multi-quadcopter system (MQS) in an obstacle-laden environment. 
% % A desired affine transformation is planned by $n+1$ leaders and acquired by the remaining follower quadcopters in real time through local communication. 
% By eigen-decomposition of the affine transformation, safe trajectories are assigned for leaders such that no two quadcopters collide and MQS containment is guaranteed while the MQS aggressively deforms. Additionally, A* search is applied to optimally plan the MQS rigid-body displacement in an obstacle-laden environment where it is guaranteed that no quadcopter hits obstacles in the motion space. The paper also proposes a proximity-based communication topology for followers to assign communication weights with their in-neighbors and acquire the desired coordination with minimal computation cost. By eigen-decomposition of the affine transformation and providing guarantee conditions for inter-agent collision avoidance and quadcopter containment, the paper advances the scalabality in a large-scale collective motion. 

\end{abstract}


\section{Introduction}\label{Introduction}
Multi-agent coordination has been an active research area over the past few decades. Many aspects of multi-agent coordination have been explored and several centralized and decentralized multi-agent control {\color{black}approaches} already exist. In spite of vast amount of existing research on multi-agent coordination, scalability, maneuverability, safety, resilience, and optimality of group coordination are still very important issues for exploration and study. The goal of this paper is to address these important problems in a formal and algorithmic way through integrating the principles of continuum mechanics, A* search method, and classic optimal control approach.

\subsection{Related Work}
Consensus and  containment control are two available decentralized muti-agent coordination approaches. Multi-agent consensus have found numerous applications such as flight formation control \cite{zhang2018collision}, multi-agent surveillance \cite{du2017pursuing}, and air traffic control \cite{artunedo2017consensus}. Consensus control of homogeneous and heterogeneous multi-agents systems \cite{cheng2018event} was studied in the past. Multi agent consensus under fixed \cite{tu2017decentralized} and switching \cite{munoz2017adaptive, liu2018leader} communication topologies have been widely investigated by the researchers over the past two decades. Stability of consensus algorithm in the presence of delay is analyzed in Ref. \cite{ma2020consensus}.  Researchers have {\color{black}also} investigated multi-agent consensus in the presence of actuation failure \cite{wang2019fault, shahab2019distributed}, sensor failure \cite{liu2017kalman}, and adversarial agents \cite{leblanc2011consensus}.

Containment control is a decentralized leader-follower multi-agent coordination approach in which the desired coordination is defined by leaders and acquired by followers through local communication. Early work studied stability and convergence of multi-agent containment protocol in Refs. \cite{ji2008containment, liu2012necessary}, under fixed \cite{li2016containment} or switching \cite{su2015multi} communication topologies, as well as multi-agent containment in the presence of fixed \cite{asgari2019necessary} and time-varying \cite{atrianfar2020sampled} time delays. Resilient containment control is studied in the presence of actuation failure \cite{cui2018command}, sensor failure \cite{ye2017observer}, and adversarial agents \cite{zuo2019resilient}. Also, researchers  investigated the problems of finite-time \cite{qin2019distributed} and fixed-time \cite{xu2020distributed} containment control of multi-agent systems in the past. 

\subsection{Contributions}
The main objective of this paper is to integrate the principles of continuum mechanics with search and optimization methods to safely plan continuum deformation of a multi-quadcopter system (MQS). In particular, we treat quadcopters as a finite number of particles of a $2$-D deformable body coordinating  in a $3$-D where the desired coordination {\color{black}of the continuum} is defined by a homogeneous deformation. Homogeneous deformation is a non-singular affine transformation which is classified as a Lagrangian continuum deformation problem. Due to linearity of homogeneous transformation, it can be defined as a decentralized leader-follower coordination problem in which leaders' desired positions are uniquely related to the components of the Jacobian matrix and rigid-body displacement vector of the homogeneous transformation at any time $t$. 

{\color{black}This paper develops} an algorithmic protocol for safe planning of coordination of a large-scale MQS by determining the global desired trajectories of leaders in an obstacle-laden motion space, containing obstacles with arbitrary geometries. To this end, we integrate the A* search method, optimal control planning, and eigen-decomposition to plan the desired trajectories of the leaders minimizing travel distances between their initial and final configurations. Containing the MQS by a rigid ball, the path of the center of the containment ball is safely determined using the A* search method. We apply the principles of Lagrangian continuum mechanics to decompose the homogeneous deformation coordination and {\color{black}to} ensure inter-agent collision avoidance {\color{black}through} constraining the deformation eigenvalues. {\color{black}By eigen-decomposition of a homogeneous transformation, we can also} determine the leaders' intermediate configurations and formally specify safety requirements for a large-scale MQS {\color{black}coordination} in a geometrically-constrained environment. Additionally, we {\color{black}assign} safe desired trajectories of leaders, connecting consecutive configurations of the leader agents, by solving a constrained optimal control planning problem.

This paper is organized as follows: Preliminary notions including graph theory definitions and position notations are presented in Section \ref{Preliminaries}. Problem Statement is presented in Section \ref{Problem Statement} and followed by continuum deformation coordination planning developed in Section \ref{Continuum Deformation Planning}. We review the existing approach for continuum deformation acquisition {\color{black}through local communication} in Section \ref{Continuum Deformation Acquisition}. Simulation results are presented in Section \ref{Simulation Results} and followed by Conclusion in Section \ref{Conclusion}.

\section{Preliminaries}\label{Preliminaries}
\subsection{Graph Theory Notions}\label{Graph Theory Notions}
We consider {\color{black}the} group coordination of a quadcopter team consisting of $N$ quadcopters in an obstacle-laden environment. Communication among quadcopters are defined by graph $\mathcal{G}\left(\mathcal{V},\mathcal{E}\right)$ with node set $\mathcal{V}=\{1,\cdots,N\}$, defining the index numbers of the quadcopters, and edge set $\mathcal{E}\subset \mathcal{V}\times \mathcal{V}$. In-neighbors of quadcopter $i\in \mathcal{V}$ is defined by set $\mathcal{N}_i=\left\{j:\left(j,i\right)\in \mathcal{E}\right\}$.

In this paper, quadcopters are treated as particles of a $2$-D continuum, where the desired coordination is defined by {\color{black}a} homogeneous transformation \cite{rastgoftar2020scalable}. {\color{black}A} desired {\color{black}$2$-D} homogeneous transformation is defined by three leaders and acquired by the remaining follower quadcopters through local communication. Without loss of generality, leaders and followers are identified by $\mathcal{V}_L=\left\{1,2,3\right\}\subset \mathcal{V}$ and $\mathcal{V}_F=\left\{4,\cdots,N\right\}$. Note that leaders move independently, therefore, $\mathcal{N}_i=\emptyset$, if $i\in \mathcal{V}_L$.

\begin{assumption}
Graph $\mathcal{G}\left(\mathcal{V},\mathcal{E}\right)$ is defined such that every follower quadcopter accesses position information of three in-ineighbor agents, thus, 
\begin{equation}
    \bigwedge_{i\in \mathcal{V}_F}\left(\mathcal{N}_i=3\right).
\end{equation}
\end{assumption}
% \begin{assumption}
% Graph $\mathcal{G}\left(\mathcal{V},\mathcal{E}\right)$ is defined such that every follower quadcopter accesses position information of three in-ineighbor agents, thus, 
% \begin{equation}
%     \bigwedge_{i\in \mathcal{V}_F}\left(\mathcal{N}_i=3\right).
% \end{equation}
% \end{assumption}

\subsection{Position Notations}
In this paper, we define actual position $\mathbf{r}_i(t)=\begin{bmatrix}
x_i(t)&y_i(t)&z_i(t)
\end{bmatrix}^T$, global desired position $\mathbf{p}_i(t)=\begin{bmatrix}
x_{i,HT}(t)&y_{i,HT}(t)&z_{i,HT}(t)
\end{bmatrix}^T$, local desired position $\mathbf{r}_{i,d}(t)=\begin{bmatrix}
x_{i,d}(t)&y_{i,d}(t)&z_{i,d}(t)
\end{bmatrix}^T$, and reference position $\mathbf{p}_{i,0}=\begin{bmatrix}
x_{i,0}&y_{i,0}&0
\end{bmatrix}^T$ for every quadcopter $i\in \mathcal{V}$. Actual position $\mathbf{r}_i(t)$ is the output vector of the control system of quadcopter $i\in \mathcal{V}$. Global desired position of quadcopter $i\in \mathcal{V}$ is defined by a homogeneous transformation with the details provided in Ref. \cite{rastgoftar2020scalable} and discussed in Section \ref{Continuum Deformation Planning}. Local desired position of quadcopter $i\in \mathcal{V}$ is given by
\begin{equation}
    \mathbf{r}_{i,d}(t)=\begin{cases}
    \mathbf{p}_i(t)&i\in \mathcal{V}_L\\
    \sum_{j\in \mathcal{N}_i}\mathbf{r}_j(t)&i\in \mathcal{V}_F\\
    \end{cases}
    ,
\end{equation}
where $w_{i,j}>0$ is a constant communication weight between follower $i\in \mathcal{V}_F$ and in-neighbor quadcopter $j\in \mathcal{N}_i$, and
\begin{equation}
    \sum_{j\in \mathcal{N}_i}w_{i,j}=1.
\end{equation}
Followers' communication weights are consistent with the reference positions of quadcopters and satisfy the following equality constraints:
\begin{equation}
    \bigwedge_{i\in \mathcal{V}_F}\left(\sum_{j\in \mathcal{N}_i}w_{i,j}\left(\mathbf{p}_{j,0}-\mathbf{p}_{i,0}\right)=0\right).
\end{equation}
\begin{remark}
The initial configuration of the MQS is obtained by a rigid-body rotation of the reference configuration. Therefore,
initial position of every quadcopter $i\in \mathcal{V}$ denoted by $\mathbf{r}_{i,s}$ is not necessarily the same as the reference position $\mathbf{p}_{i,0}${\color{black},} but $\mathbf{r}_{i,s}$ and $\mathbf{p}_{i,0}$ satisfy the following relation:
\begin{equation}
    \bigwedge_{i=1}^{N-1}\bigwedge_{j=i+1 }^N\left(\|\mathbf{r}_{i,s}-\mathbf{r}_{j,s}\|=\|\mathbf{p}_{i,0}-\mathbf{p}_{j,0}\|\right),
\end{equation}
where $\|\cdot\|$ is the 2-norm symbol.
\end{remark}
\section{Problem Statement}\label{Problem Statement}

We treat the MQS as {\color{black}particles} of a {\color{black}$2$-D} deformable body navigating in an obstacle-laden environment. 
% The geometries of obstacles containing obstacles are identified by vector $\mathbf{o}=\begin{bmatrix}
% \mathbf{o}_1^T&\cdots&\mathbf{o}_M^T\end{bmatrix}^T$, where $\mathbf{o}_j$ defines vertices of the polytope containing obstacle $j\in \mathcal{P}$, i.e. the identification number $j\in \mathcal{P}$ defines the identification number of the $j$-th obstacle and its enclosing polytope. 
The desired formation of the MQS is given by
\begin{equation}\label{followerleaderdesiredrelation}
    \mathbf{y}_{F,HT}(t)=\mathbf{H}\mathbf{y}_{L,HT}(t),
\end{equation}
at any time $t\in [t_s,t_u]$, where $\mathbf{H}\in \mathbb{R}^{3\left(N-3\right)\times 9}$ is a constant shape matrix that is obtained based on reference positions in Section \ref{Continuum Deformation Planning}{\color{black}. Also,}
\begin{subequations}
\begin{equation}
    \mathbf{y}_{L,HT}=\mathrm{vec}\left(\begin{bmatrix}
\mathbf{p}_{1}&\cdots&\mathbf{p}_{3}
\end{bmatrix}^T\right)\in \mathbb{R}^{9\times 1},
\end{equation}
\begin{equation}
    \mathbf{y}_{F,HT}=\mathrm{vec}\left(\begin{bmatrix}
\mathbf{p}_{4}&\cdots&\mathbf{p}_{N}
\end{bmatrix}^T\right)\in \mathbb{R}^{3\left(N-3\right)\times 1}
\end{equation}
\end{subequations}
 aggregate the components of desired positions of followers and leaders, respectively, where ``vec'' {\color{black}is} the matrix vectorization symbol. Per Eq. \eqref{followerleaderdesiredrelation}, {\color{black}the} desired formation of followers{\color{black}, assigned by $\mathbf{y}_{F,HT}(t)$,}  is uniquely determined based on the desired leaders' trajectories defined by $\mathbf{y}_{L,HT}(t)$ over the time interval $\left[t_s,t_u\right]$. 
 The MQS is constrained to remain inside the rigid  containment ball
\begin{equation}
\resizebox{0.99\hsize}{!}{%
$
    \mathcal{S}\left({\mathbf{d}}\left(t\right),r_{\mathrm{max}}\right)=\left\{\left(x,y,z\right):\left(x-d_x\right)^2+\left(x-d_x\right)^2+\left(z-d_z\right)^2\leq r_{\mathrm{max}}^2\right\}
    $
    }
\end{equation}
with {\color{black}the} constant radius $r_{\mathrm{max}}$ and {\color{black}the} center $\mathbf{d}(t)=\begin{bmatrix}
d_x(t)&d_y(t)&d_z(t)
\end{bmatrix}^T$ at time $t\in \left[t_s,t_u\right]$. 

The main objective of this paper is to determine $\mathbf{y}_{L,HT}(t)$ and ultimate time $t_u$ such that the MQS travel distances are minimized, and the following constraints are all satisfied at any time $t\in \left[t_s,t_u\right]$:
\begin{subequations}
\begin{equation}\label{c1}
\forall t\in \left[t_s,t_u\right],\qquad \mathbf{y}_{L,HT}^T\left(t\right)\mathbf{\Psi}\mathbf{y}_{L,HT}\left(t\right)-A_s=0,
\end{equation}
\begin{equation}\label{c2}
\forall t\in \left[t_s,t_u\right],\qquad \bigwedge_{i\in \mathcal{V}}\bigwedge_{j\in \mathcal{V},j\neq i}\|\mathbf{r}_i(t)-\mathbf{r}_j(t)\|\neq 2\epsilon,
\end{equation}
% \begin{equation}
% \forall t\in \left[t_s,t_u\right],\qquad \bigwedge_{i\in \mathcal{V}}\bigwedge_{j\in \mathcal{V},j\neq i}\|\mathbf{r}_i(t)-\mathbf{r}_j(t)\|\neq 2\epsilon,
% \end{equation}
\begin{equation}\label{c3}
\forall t\in \left[t_s,t_u\right],\qquad \bigwedge_{i\in \mathcal{V}_L} \left(z_{i,HT}\left(t\right)=d_z(t)\right),
\end{equation}
\begin{equation}\label{c4}
\forall t\in \left[t_s,t_u\right],\qquad \bigwedge_{\in \mathcal{V}}\left(\left(x_{i,HT}(t),y_{i,HT}(t)\right)\in \mathcal{S}\left({\color{black}\mathbf{d}}\left(t\right),r_{\mathrm{max}}\right)\right),
\end{equation}
\end{subequations}
where {\color{black}$x_{i,HT}(t)$ and $y_{i,HT}(t)$ are the $x$ and $y$ components of the global desired position of quadcopter $i\in \mathcal{V}$ at time $t\in \left[t_s,t_u\right]$,}
\begin{equation}
    \mathbf{\Psi}=\mathbf{O}^T\mathbf{P}^T\mathbf{P}\mathbf{O}
\end{equation}
is constant, 
\begin{subequations}
\begin{equation}
    \mathbf{P}={1\over 4}\begin{bmatrix}
    0&0&0&0&1& -1\\
    0&0&0&-1&0&1\\
    0&0&0&1&-1&0\\
    0&-1&1& 0&0&0\\
     1&0&-1& 0&0&0\\
      -1&1&0& 0&0&0\\
    % \mathbf{0}&\mathbf{K}_a\\
    % \mathbf{K}_a^T&\mathbf{0}&
    \end{bmatrix}
    ,
\end{equation}
and $\mathbf{O}=\begin{bmatrix}
\mathbf{I}_6&\mathbf{0}_{6\times 3}
\end{bmatrix}$.
\end{subequations}
The constraint equation \eqref{c1} ensures that the area of the leading triangle, with vertices occupied by the desired position of the leaders, remains constant and equal to $A_s$ at any time $t\in \left[t_s,t_u\right]$. 
Constraint equation \eqref{c2} ensures that no two quadcopters collide{\color{black},} if every quadcopter {\color{black}$i\in \mathcal{V}$} can be enclosed by a ball with constant radius $\epsilon$. Constraint equation \eqref{c3} ensures that the desired formation of the MQS lies in a horizontal plane at any time $t\in \left[t_s,t_u\right]$. Per Eq. \eqref{c4}, the desired MQS formation is {\color{black}constrained} to remain inside the ball $\mathcal{S}\left({\mathbf{d}}(t),r_{\mathrm{max}}\right)$ at any time $t\in \left[t_s,t_u\right]$.

{\color{black}To accomplish the goal of this paper, we} integrate (i) A* search, (ii) eigen-decomposition, (iii) optimal control planning to assign leaders' optimal trajectories ensuring safety requirements \eqref{c1}-\eqref{c4} {\color{black}by performing the following sequential steps:}

\textbf{Step 1: Assigning Intermediate Locations of the Containment Ball:} Given initial and final positions of the center of the containment ball, denoted by {\color{black}$\bar{\mathbf{d}}_s=\mathbf{d}\left(t_s\right)=\bar{\mathbf{d}}_0$ and $\bar{\mathbf{d}}_u=\mathbf{d}\left(t_s\right)=\bar{\mathbf{d}}_{n_\tau}$}, and obstacle geometries, we apply the A* search method to determine the intermediate positions of the center of the containment ball $\mathcal{S}$, denoted by $\bar{\mathbf{d}}_{1}$, $\cdots$, $\bar{\mathbf{d}}_{n_\tau-1}$, such that: (i) the travel distance between the initial and final configurations of the MQS is minimized and (ii) the containment ball do not collide the obstacles{\color{black},} arbitrarily distributed in the coordination space. 

{\color{black}\textbf{Step 2: Assigning Leaders' Intermediate Configurations:} By knowing $\bar{\mathbf{d}}_{1}$, $\cdots$, $\bar{\mathbf{d}}_{n_\tau-1}$, we define
\begin{equation}\label{betak}
\beta_k= \dfrac{\sum_{j=0}^k\left(\bar{\mathbf{d}}_{j}-\bar{\mathbf{d}}_{0}\right)}{\sum_{j=0}^{n_\tau}\left(\bar{\mathbf{d}}_{j}-\bar{\mathbf{d}}_{0}\right)}
\end{equation}
and
\begin{equation}\label{tku}
    k=0,1,\cdots,n_\tau,\qquad t_k(t_u)=\left(1-\beta_k\right)t_s+\beta_kt_u
\end{equation}
for   $k=0,\cdots,n_\tau$, where $t_k$ is when the center of the containment ball $\mathcal{S}$ reaches desired intermediate position $\bar{\mathbf{d}}_{k}$.} Given $\mathbf{y}_{L,HT}\left(t_s\right)=\bar{\mathbf{y}}_{L,h,0}$, $\mathbf{y}_{L,HT}\left(t_u\right)=\bar{\mathbf{y}}_{L,h,n_\tau}$, Section \ref{Planning2} decomposes the homogeneous deformation coordination to determine 
% initial and ultimate positions of the center of the containment ball, denoted by  $\mathbf{d}\left(t_s\right)=\bar{\mathbf{d}}_{0}$ and $\mathbf{d}\left(t_u\right)=\bar{\mathbf{d}}_{u}$, and 
the intermediate configurations of the leaders {\color{black}that are denoted by $\bar{\mathbf{y}}_{L,HT,1}$, $\cdots$, $\bar{\mathbf{y}}_{L,HT,n_\tau-1}$}. 
% ,  $t_u=t_{n_\tau}$, and

% Section \ref{Planning2} determines intermediate configurations $\bar{\mathbf{y}}_{L,h,k}$ for $k=1,\cdots,n_\tau-1$.  
\textbf{Step 3: Assigning Leaders' Desired Trajectories:} By expressing $\bar{\mathbf{d}}=\begin{bmatrix}
\bar{d}_{x,k}&\bar{d}_{y,k}&\bar{d}_{z,k}
\end{bmatrix}^T$ for $k=0,1,\cdots,n_\tau$, $z$ components of the leaders' desired trajectories are {\color{black}the same at anytime $t\in \left[t_s,t_u\right]$}, and defined by
\begin{equation}\label{zcomponent}
    \forall i\in \mathcal{V}_L,\qquad z_{i,HT}=\bar{d}_{z,k}\left(1-\gamma(t,T_k)\right)+\bar{d}_{z,k+1}\gamma\left(t,T_k\right)
\end{equation}
at any time $t\in \left[t_k,t_{k+1}\right]$ for $k=0,\cdots,n_\tau-1$, where $T_k=t_{k+1}-t_k$, and
\begin{equation}\label{gamma}
    \gamma(t,T_k)=6\left({t-t_k\over t_{k+1}-t_k}\right)^5-15\left({t-t_k\over t_{k+1}-t_k}\right)^4+10\left({t-t_k\over t_{k+1}-t_k}\right)^3
\end{equation}
for $t\in \left[t_k,t_{k+1}\right]$. 
Note that $\gamma(t_k)=0$, $\gamma_{k+1}=1$, $\dot{\gamma}\left(t_k\right)=\dot{\gamma}\left(t_{k+1}\right)=0$, and $\ddot{\gamma}\left(t_k\right)=\ddot{\gamma}\left(t_{k+1}\right)=0$.

The $x$ and $y$ components of the desired trajectories of leaders are governed by dynamics
\begin{equation}\label{maindynamicsss}
\dot{\mathbf{x}}_L=\mathbf{A}_L{\mathbf{x}}_L+\mathbf{B}_L{\mathbf{u}}_L,
\end{equation}
where ${\mathbf{u}}_L\in \mathbf{R}^{9\times 1}$ is the input vector, and 
\begin{subequations}
\begin{equation}
\mathbf{x}_L(t)=\left(\mathbf{I}_2\otimes\mathbf{O}\right)\begin{bmatrix}
\mathbf{y}_{L,HT}(t)\\\dot{\mathbf{y}}_{L,HT}(t)
\end{bmatrix}
^T\in \mathbb{R}^{12\times 1}
\end{equation}
\begin{equation}
    \mathbf{A}_L=\begin{bmatrix}
    \mathbf{0}_{6\times 6}&\mathbf{I}_{6}\\
    % \mathbf{0}_{9}&\mathbf{0}_{9\times 9}&\mathbf{I}_9\\
    \mathbf{0}_{6\times 6}&\mathbf{0}_{6\times 6}\\
    \end{bmatrix}
    ,
\end{equation}
\begin{equation}
    \mathbf{B}_L=\begin{bmatrix}
    \mathbf{0}_{6\times 6}\\
    % \mathbf{0}_{9}\\
    \mathbf{I}_{6}\\
    \end{bmatrix}
    ,
\end{equation}
\end{subequations}
$\mathbf{0}_{6\times 6}\in \mathbb{R}^{6\times 6}$ is a zero-entry matrix, and $\mathbf{I}_6\in \mathbb{R}^{6\times 6}$ is an identity matrix. 
Control input $\mathbf{u}_L\in \mathbb{R}^{6\times 1}$ is optimized by minimizing cost function
\begin{equation}
   \min \mathrm{J}(\mathbf{u}_L,t_u)=\min {1\over 2}\sum_{k=0}^{n_\tau-1} \left(\int_{t_k(t_u)}^{t_{k+1}(t_u)}\mathbf{u}_L^T\left(\tau\right)\mathbf{u}_L\left(\tau\right)d\tau\right)
\end{equation}
subject to dynamics \eqref{maindynamicsss}, safety conditions \eqref{c1}-\eqref{c4}, and boundary conditions
\begin{equation}
    \bigwedge_{k=0}^{n_\tau}\left(\mathbf{x}_{L}(t_k)=\bar{\mathbf{x}}_{L,k}\right).
\end{equation}

A desired continuum deformation coordination, planned by the leader quadcopters, is acquired by followers in a decentralized fashion using the protocol developed in Refs. \cite{rastgoftar2020scalable, rastgoftar2020fault}. This protocol is discussed in Section  \ref{Continuum Deformation Acquisition}.

% Given 

% the A* search is applied to determine the intermediate leaders' configurations, denoted by $\bar{\mathbf{y}}_{L,h,1}$, $\cdots$, $\bar{\mathbf{y}}_{L,h,n_\tau-1}$ via minimizing the MQS travel distance in an obstacle-laden environment, where
% \begin{equation}
%     k=0,\cdots,n_\tau,\qquad \mathbf{y}_{L,HT}(t_k)=\mathbf{y}_{L,h,k}
% \end{equation}
% and $t_{n_\tau}=t_u$.

% . Given two consecutive leaders waypoiyts $\bar{\mathbf{y}}_{L,h,k}$ and $\bar{\mathbf{y}}_{L,h,k+1}$, the optimal desired trajectories of leaders are determined by solving 

\section{Continuum Deformation Planning}\label{Continuum Deformation Planning}
The desired configuration of the MQS is defined by {\color{black}affine transformation}
\begin{equation}\label{globaldesiredcoordination}
i\in \mathcal{V},\qquad     \mathbf{p}_i\left(t\right)=\mathbf{Q}\left(t\right){\mathbf{p}}_{i,0}+\mathbf{s}\left(t\right),
\end{equation}
at time $t\in \left[t_s,t_u\right]$, where $\mathbf{p}_i(t)=\begin{bmatrix}
x_{i,HT}(t)&y_{i,HT}(t)&z_{i,HT}(t)
\end{bmatrix}^T\in \mathbb{R}^3$ is the desired position of quadcopter $i\in \mathcal{V}$, $\mathbf{p}_{i,0}$ is the reference position of quadcopter $i\in \mathcal{V}$, and  $\mathbf{s}(t)=\begin{bmatrix}
s_x(t)&s_y(t)&s_z(t)
\end{bmatrix}^T$ is the rigid body displacement vector.
% defining the center of containment ball $\mathcal{S}$ at time $t$, and $\bar{\mathbf{d}}_s=\mathbf{d}\left(t_s\right)$.
Also, Jacobian matrix $\mathbf{Q}=\left[Q_{ij}\right]\in \mathbb{R}^{3\times 3}$ given by
\begin{equation}
    \mathbf{Q}\left(t\right)=\begin{bmatrix}
    \mathbf{Q}_{xy}(t)&\mathbf{0}_{2\times 1}\\
    \mathbf{0}_{1\times 2}&1
    \end{bmatrix}
\end{equation}
is non-singular at any time $t\in \left[t_s,t_u\right]${\color{black},} where $\mathbf{Q}_{xy}(t)\in \mathbb{R}^{2\times 2}$ specifies the deformation of the leading triangle, defined by the three leaders. Because $Q_{31}=Q_{32}=Q_{13}=Q_{23}=0$, the leading triangle lies in the horizontal plane at any time $t\in \left(t_s,t_u\right]$, if the $z$ components of desired positions of the leaders are all identical at the initial time $t_s$. 
\begin{assumption}\label{assmq1}
This paper assumes that $\mathbf{Q}(t_s)=\mathbf{I}_3$. Therefore, initial and reference positions of quadcopter $i\in \mathcal{V}$ are related by
\begin{equation}
    \mathbf{p}_i\left(t_s\right)=\mathbf{p}_{i,0}+\bar{\mathbf{d}}_s.
\end{equation}
\end{assumption}

The global desired trajectory of quadcopter $i\in \mathcal{V}${\color{black}, defined by affine transformation \eqref{globaldesiredcoordination}, can be expressed} by
\begin{equation}\label{form2}
    \mathbf{p}_i(t)=\left(\mathbf{I}_3\otimes \mathbf{\Omega}_2^T\left(\mathbf{p}_{1,0},\mathbf{p}_{2,0},\mathbf{p}_{3,0},\mathbf{p}_{i,0}\right)\right)\mathbf{y}_{L,HT}{\color{black}\left(t\right)},
\end{equation}
where $\mathbf{\Omega}_2\left(\mathbf{p}_{1,0},\mathbf{p}_{2,0},\mathbf{p}_{3,0},\mathbf{p}_{i,0}\right)\in \mathbb{R}^{3\times 1}$ is defined based on reference positions of leaders $1$, $2$, and $3${\color{black}, as well as} quadcopter $i\in \mathcal{V}$ by
\begin{equation}
    \mathbf{\Omega}_2\left(\mathbf{p}_{1,0},\mathbf{p}_{2,0},\mathbf{p}_{3,0},\mathbf{p}_{i,0}\right)=\begin{bmatrix}
    x_{1,0}&x_{2,0}&x_{3,0}\\
    y_{1,0}&y_{2,0}&y_{3,0}\\
    1&1&1
    \end{bmatrix}
    ^{-1}
    \begin{bmatrix}
    x_{i,0}\\
    y_{i,0}\\
    1
    \end{bmatrix}
    .
\end{equation}
Note that sum of the entries of vector $\mathbf{\Omega}_2\left(\mathbf{p}_{1,0},\mathbf{p}_{2,0},\mathbf{p}_{3,0},\mathbf{p}_{i,0}\right)$ {\color{black}is $1$} for arbitrary {\color{black}vectors} $\mathbf{p}_{1,0}$, $\mathbf{p}_{2,0}$, $\mathbf{p}_{3,0}$, and $\mathbf{p}_{i,0}$, distributed in the $x-y$ plane, if  $\mathbf{p}_{1,0}$, $\mathbf{p}_{2,0}$, $\mathbf{p}_{3,0}$ form a triangle.
\begin{remark}
By using Eq. \eqref{form2}, followers' global desired positions can be expressed based on leaders' global desired positions using relation \eqref{followerleaderdesiredrelation}, where
\begin{equation}
    \mathbf{H}=\mathbf{I}_3\otimes \begin{bmatrix}
     \mathbf{\Omega}_2^T\left(\mathbf{p}_{1,0},\mathbf{p}_{2,0},\mathbf{p}_{3,0},\mathbf{p}_{4,0}\right)\\
     \vdots\\
     \mathbf{\Omega}_2^T\left(\mathbf{p}_{1,0},\mathbf{p}_{2,0},\mathbf{p}_{3,0},\mathbf{p}_{N,0}\right)\\
     \end{bmatrix}
     \in \mathbb{R}^{3\left(N-3\right)\times 9}
\end{equation}
is constant and determined based on reference positions of {\color{black}the} MQS.
\end{remark}

\begin{remark}
Eq. \eqref{globaldesiredcoordination} is used for eigen-decomposition, safety analysis, and planning of the desired continuum deformation coordination. On the other hand, Eq. \eqref{form2} is used in Section \ref{Communication-Based Guidance Protocol} to define the MQS continuum as a decentralized leader-follower problem and ensure the boundedness of the trajectory tracking controllers {\color{black}that are} independently planned by individual {\color{black}quadcopeters}.
\end{remark}
% The $x$ and $y$ components of the leaders' desired trajectories, aggregated by vector $\mathbf{q}_g=\begin{bmatrix}
% x_{1,HT}&x_{2,HT}&x_{3,HT}&y_{1,HT}&y_{2,HT}&y_{3,HT}
% \end{bmatrix}^T$, can be related to the elements of $\mathbf{Q}_{xy}$ ($Q_{11}$, $Q_{12}$, $Q_{21}$, and $Q_{22}$) and $d_x$ and $d_y$ components of vector $\mathbf{d}$ at any time $t$ by \cite{rastgoftar2020fault}.
\begin{theorem}\label{thm1}
Assume that three leader quadcopters $1$, $2$, and $3$ remain non-aligned at any time $t\in [t_s,t_u]$. Then, the desired configuration of {\color{black}the} leaders at time $t\in \left[t_s,t_u\right]$, defined by $\mathbf{y}_{L,HT}(t)$, is related to the leaders' initial configuration, defined by $\bar{\mathbf{y}}_{L,HT,0}$, and the rigid body displacement vector $\mathbf{s}(t)$ by
\begin{equation}
    \mathbf{y}_{L,HT}(t)=\mathbf{D}\left(\mathbf{I}_3\otimes \mathbf{Q}(t)\right)\mathbf{D}\bar{\mathbf{y}}_{L,HT,0}
    % \mathbf{D}\left(\mathbf{I}_3\otimes \mathbf{Q}(t)\right)\mathbf{D}+
    % \mathbf{D}\left(\mathbf{1}_{3\times 1}\otimes 
   +\mathbf{D}\left(\mathbf{1}_{3\times 1}\otimes \mathbf{s}(t)\right),
\end{equation}
where $\otimes$ is the Kronecker product symbol and $\mathbf{D}\in \mathbb{R}^{9\times 9}$ is an involutory matrix defined as follows:
\begin{equation}
    D_{ij}=\begin{cases}
        1&i=1,2,3,~j=3(i-1)+1\\
        1&i=4,5,6,~j=3(i-1)+2\\
        1&i=7,6,9,~j=3i\\
    \end{cases}
    .
\end{equation}
% \begin{subequations}
% \begin{equation}
%     \mathbf{\Gamma}\left(t\right)=\mathbf{D}\left(\mathbf{I}_3\otimes \mathbf{Q}(t)\right)\mathbf{D},
% \end{equation}
% \begin{equation}
%     \mathbf{b}\left(t\right)=\mathbf{D}\left(\mathbf{1}_{3\times 1}\otimes \mathbf{d}(t)\right),
% \end{equation}
% \end{subequations}
% and 
 Also, elements of matris $\mathbf{Q}_{xy}{\color{black}\left(t\right)}$
% vector $\mathbf{e}(t)=\begin{bmatrix}
% Q_{11}&Q_{12}&Q_{21}&Q_{22}
% \end{bmatrix}^T$ aggregating the elements of $\mathbf{Q}_{xy}$ 
% ($Q_{11}$, $Q_{12}$, $Q_{21}$, and $Q_{22}$) 
and rigid-body displacement vector $\mathbf{s}(t)$ can be related to $\mathbf{y}_{L,HT}(t)$ by 
\begin{subequations}
\begin{equation}\label{qelement}
Q_{11}(t)=\mathbf{E}_1\mathbf{\Gamma}\mathbf{O}\mathbf{y}_{L,HT}(t),
\end{equation}
\begin{equation}\label{qelement}
Q_{12}(t)=\mathbf{E}_2\mathbf{\Gamma}\mathbf{O}\mathbf{y}_{L,HT}(t),
\end{equation}
\begin{equation}\label{qelement}
Q_{21}(t)=\mathbf{E}_3\mathbf{\Gamma}\mathbf{O}\mathbf{y}_{L,HT}(t),
\end{equation}
\begin{equation}\label{qelement}
Q_{22}(t)=\mathbf{E}_4\mathbf{\Gamma}\mathbf{O}\mathbf{y}_{L,HT}(t),
\end{equation}
\begin{equation}\label{dt}
\mathbf{s}(t)=\begin{bmatrix}
\mathbf{E}_5\mathbf{\Gamma}\mathbf{O}\\
\mathbf{E}_6
\end{bmatrix}
\mathbf{y}_{L,HT}(t),
\end{equation}
\end{subequations}
{\color{black}at any time $t\in \left[t_s,t_u\right]$,} where $\mathbf{E}_1=\begin{bmatrix}
1&\mathbf{0}_{1\times 5}
\end{bmatrix}$, $\mathbf{E}_2=\begin{bmatrix}
0&1&\mathbf{0}_{1\times 4}
\end{bmatrix}$, $\mathbf{E}_3=\begin{bmatrix}
\mathbf{0}_{1\times 2}&1&\mathbf{0}_{1\times 3}
\end{bmatrix}$, $\mathbf{E}_4=\begin{bmatrix}
\mathbf{0}_{1\times 3}&1&\mathbf{0}_{1\times 2}
\end{bmatrix}$, $\mathbf{E}_5=\begin{bmatrix}
\mathbf{0}_{2\times 4}&\mathbf{I}_2
\end{bmatrix}$, $\mathbf{E}_6={1\over 3}\begin{bmatrix}
\mathbf{0}_{1\times 6}&\mathbf{1}_{1\times 3}
\end{bmatrix}\in \mathbb{R}^{3\times 9}$, and
\[
\mathbf{\Gamma}=\begin{bmatrix}
    x_{1,0}&y_{1,0}&0&0&1&0\\
    x_{2,0}&y_{2,0}&0&0&1&0\\
    x_{3,0}&y_{3,0}&0&0&1&0\\
    0&0&x_{1,0}&y_{1,0}&0&1\\
    0&0&x_{2,0}&y_{2,0}&0&1\\
    0&0&x_{3,0}&y_{3,0}&0&1\\
    \end{bmatrix}
    ^{-1}
    .
\]

% and $d_x$ and $d_y$ components of vector $\mathbf{d}$ can be related to  the $x$ and $y$ components of the leaders' desired trajectories, aggregated by vector $\mathbf{q}_g=\begin{bmatrix}
% x_{1,HT}&x_{2,HT}&x_{3,HT}&y_{1,HT}&y_{2,HT}&y_{3,HT}
% \end{bmatrix}^T$, 
% by 
% \begin{equation}
%     \mathbf{e}(t)=\begin{bmatrix}
%     \mathbf{I}_2\otimes \mathbf{L}_s&\mathbf{I}_2\otimes \mathbf{1}_{3\times 1}
%     \end{bmatrix}
%     ,
% \end{equation}

% \begin{subequa}
% \begin{equation}
%     \mathbf{\Gamma}(t)=\mathbf{D}\left(\mathbf{I}_3\otimes \mathbf{Q}\right)\mathbf{D}
% \end{equation}

\end{theorem}
\begin{proof}
Vectors $\mathbf{y}_{L,HT}{\color{black}\left(t\right)}$ and $\bar{\mathbf{y}}_{L,HT,0}$ can be expressed by $\mathbf{y}_{L,HT}(t)=\mathbf{D}\begin{bmatrix}
\mathbf{p}_1^T(t)&\mathbf{p}_2^T(t)&\mathbf{p}_3^T(t)
\end{bmatrix}^T$ and $\bar{\mathbf{y}}_{L,HT,0}=\mathbf{D}\begin{bmatrix}
\mathbf{p}_{1,0}&\mathbf{p}_{2,0}&\mathbf{p}_{3,0}
\end{bmatrix}^T$, respectively. By provoking Eq. \eqref{globaldesiredcoordination},
% The desired position of leader $i\in \mathcal{V}_L$ is defined by Eq. \eqref{globaldesiredcoordination},
we can write
\begin{equation}\label{prooofffff2}
    \begin{bmatrix}
\mathbf{p}_1(t)\\
\mathbf{p}_2(t)\\
\mathbf{p}_3(t)
\end{bmatrix}
=\left(\mathbf{I}_3\otimes \mathbf{Q}(t)\right)
\begin{bmatrix}
\mathbf{p}_{1,0}\\
\mathbf{p}_{2,0}\\
\mathbf{p}_{3,0}
\end{bmatrix}
+\mathbf{1}_{3\times 1}\otimes \mathbf{s}(t){\color{black},}
\end{equation}
{\color{black}and} Eq. \eqref{prooofffff2} can be rewritten as follows:
\begin{equation}\label{proofeq1}
    \mathbf{D}\mathbf{y}_{L,HT}(t)=\mathbf{I}_3\otimes  \mathbf{D}\bar{\mathbf{y}}_{L,HT,0}+\mathbf{1}_{3\times 1}\otimes \mathbf{d}(t).
\end{equation}
Because  $\mathbf{D}$ is involutory, $\mathbf{D}=\mathbf{D}^{-1}$ and Eq. \eqref{globaldesiredcoordination} can be obtained by pre-multiplying $\mathbf{D}$ on both sides of Eq. \eqref{prooofffff2}. By replacing $\mathbf{p}_i(t)$ and $\mathbf{p}_{i,0}$ by {\color{black}$\begin{bmatrix}
x_{i,HT}(t)&y_{i,HT}(t)&z_{i,HT}(t)
\end{bmatrix}^T$ and $\begin{bmatrix}
x_{i,0}&y_{i,0}(t)&0
\end{bmatrix}^T$ into Eq. \eqref{globaldesiredcoordination} for every leader $i\in \mathcal{V}_L$, elements of $\mathbf{Q}_{xy}(t)$, denoted by $Q_{11}(t)$, $Q_{12}(t)$, $Q_{21}(t)$, and $Q_{22}(t)$, and $x$ and element of $\mathbf{s}(t)$, denoted by $s_x(t)$, and $s_y(t)$,} can be related to the $x$ and $y$ components of the leaders' desired positions 
\[
\begin{bmatrix}
Q_{11}&Q_{12}&Q_{21}&Q_{22}&s_x&s_y
\end{bmatrix}
^T
=\mathbf{\Gamma}\mathbf{O}\mathbf{y}_{L,HT},
\]
{\color{black}at any time $t\in \left[t_s,t_u\right]$,} where 
\[
\mathbf{O}\mathbf{y}_{L,HT}=\begin{bmatrix}
x_{1,HT}&x_{2,HT}&x_{3,HT}&y_{1,HT}&y_{2,HT}&y_{3,HT}
\end{bmatrix}^T.
\]
Note that matrix $\mathbf{\Gamma}$ is non-singular{\color{black},} if leaders are non-aligned at the initial time $t_s$ \cite{rastgoftar2020scalable}.
\end{proof}
Theorem \ref{thm1} is used in Section \ref{Planning1} to obtain the final location of the center of the containment ball, denoted by $\bar{\mathbf{d}}_u$, where $\bar{\mathbf{d}}_u$ is one of the inputs of the A* solver (See Algorithm \ref{euclid}). In particular, $\bar{\mathbf{d}}_u=\mathbf{s}\left(t_u\right)$ is obtained by Eq. \eqref{dt}{\color{black}, if} $\mathbf{y}_{L,HT}(t)$ is substituted by $\bar{\mathbf{y}}_{L,HT,n_\tau}=\mathbf{y}_{L,HT}(t_u)$ on the right-hand side of Eq. \eqref{dt}. In addition, Section \ref{Planning2} uses Theorem \ref{thm1} to assign the intermediate formations of the leader team.
% \begin{equation}
%     \bar{\mathbf{d}}_u=
% \end{equation}
% \begin{equation}
    
% \end{equation}
% \begin{remark}
% The $x$ and $y$ components of the leaders' desired trajectories, aggregated by vector $\mathbf{q}_g=\begin{bmatrix}
% x_{1,HT}&x_{2,HT}&x_{3,HT}&y_{1,HT}&y_{2,HT}&y_{3,HT}
% \end{bmatrix}^T$, can be related to $\mathbf{y}_{L,HT}$ at any time $t$ by
% \begin{equation}
%     \mathbf{q}_g\left(t\right)=\mathbf{O}\mathbf{y}_{L,HT}(t).
% \end{equation}
% where $\mathbf{O}=\begin{bmatrix}
% \mathbf{I}_6&\mathbf{0}_{6\times 3}
% \end{bmatrix}$.
% \end{remark}

% \begin{equation}
%     \mathbf{y}_{L,HT}
% \end{equation}

% By knowing $\bar{\mathbf{d}}_{1}$, $\cdots$, $\bar{\mathbf{d}}_{n_\tau-1}$, we assign 
% \begin{equation}
%     k=0,1,\cdots,n_\tau,\qquad t_k(t_u)=\left(1-\beta_k\right)t_s+\beta_kt_u
% \end{equation}
% where $t_k$ is when the center of the containment ball $\mathcal{S}$ reaches the $\bar{\mathbf{d}}_{k}$,  $t_u=t_{n_\tau}$, and
% \begin{equation}
%   k=0,\cdots,n_\tau,\qquad  \beta_k= \dfrac{\sum_{j=0}^k\left(\bar{\mathbf{d}}_{j}-\bar{\mathbf{d}}_{0}\right)}{\sum_{j=0}^{n_\tau}\left(\bar{\mathbf{d}}_{j}-\bar{\mathbf{d}}_{0}\right)}
% \end{equation}

% Given $\beta_s$, $\cdots$, $\beta_{n_\tau}$, $\bar{\mathbf{y}}_{L,h,0}$, $\bar{\mathbf{y}}_{L,h,n_\tau}$, Section \ref{Planning2} determines intermediate configurations $\bar{\mathbf{y}}_{L,h,k}$ for $k=1,\cdots,n_\tau-1$. 
\subsection{A* Search Planning}
\label{Planning1}
The A* search method is used to safely plan the coordination of the containment disk $\mathcal{S}$ by optimizing the intermediate {\color{black}locations of the center of the containment ball, denoted by} $\bar{\mathbf{d}}_1$ through $\bar{\mathbf{d}}_{n_\tau-1}$, {\color{black}for} given $\bar{\mathbf{d}}_s$ and $\bar{\mathbf{d}}_{n_\tau}$, {\color{black}where geometry of obstacles is known in the coordination space}. We first develop an algorithm for collision avoidance of the MQS with obstacles in Section \ref{Obstacle Collision Avoidance}. This algorithm is used by the A* optimizer to determine $\bar{\mathbf{d}}_1$ through $\bar{\mathbf{d}}_{n_\tau-1}$, as described in Section  \ref{A* Optimizer Functionality}.
\begin{definition}
Let $i-j-k-l$ be an arbitrary tetrahedron whose vertices are positioned as  $\mathbf{p}_i=\begin{bmatrix}
x_i&y_i&z_i
\end{bmatrix}^T$, $\mathbf{p}_j=\begin{bmatrix}
x_j&y_j&z_j
\end{bmatrix}^T$, $\mathbf{p}_k=\begin{bmatrix}
x_k&y_k&z_k
\end{bmatrix}^T$, and $\mathbf{p}_l=\begin{bmatrix}
x_l&y_l&z_l
\end{bmatrix}^T$ is a $3$-D coordination space. Also, $\mathbf{p}_f=\begin{bmatrix}
x_f&y_f&z_f
\end{bmatrix}^T$ is the position of an arbitrary point $f$ in the coordination space. Then,
\begin{equation}
\mathbf{\Omega}_3\left(\mathbf{p}_i,\mathbf{p}_j,\mathbf{p}_k,\mathbf{p}_l,\mathbf{p}_f\right)=\begin{bmatrix}
\mathbf{p}_i&\mathbf{p}_j&\mathbf{p}_k&\mathbf{p}_l\\
1&1&1&1
\end{bmatrix}
^{-1}
\begin{bmatrix}
\mathbf{p}_f\\
1
\end{bmatrix}
\end{equation}
is a finite vector with the entries summing up to $1$ \cite{rastgoftar2020scalable}.
\end{definition}
The vector function $\mathbf{\Omega}_3$ is used in Section \ref{Obstacle Collision Avoidance} to {\color{black}specify collision avoidance condition.}
\subsubsection{Obstacle Collision Avoidance}
\label{Obstacle Collision Avoidance}
We enclose obstacles by a finite number of polytopes identified by set ${\color{black}\mathcal{H}}=\left\{1,\cdots,M\right\}$, where
$
    \mathcal{P}=\bigcup_{j\in {\color{black}\mathcal{H}}}\mathcal{P}_j
$
defines vertices of polytopes containing obstacles in the motion space, and $\mathcal{P}_j$ is a finite set defining identification numbers of vertices {\color{black}of} polytope $j\in \mathcal{O}$ containing the  $j-th$ obstacle in the motion space. Polytope $\mathcal{P}_j$ is made of $m_j$ distinct tetrahedral cells, where $\mathcal{T}_{j,l}$ defines the identification numbers of the nodes of the $l$-th tetrahedral cell ($l=1,\cdots,m_j$). Therefore, $\mathcal{P}$ can be expressed as follows:
\begin{equation}
    \mathcal{P}=\bigwedge_{j\in \mathcal{P}}\bigwedge_{l=1}^{m_j}\mathcal{T}_{j,l}.
\end{equation}

\begin{definition}
We say ${\mathbf{d}}$ is a valid position for the center of the containment ball $\mathcal{S}$ with radius $r_{\mathrm{max}}$, if the following two conditions are satisfied:
\begin{subequations}
\begin{equation}\label{firstcontainmentcond}
    \bigwedge_{j\in \mathcal{P}}\bigwedge_{l=1}^{m_j}\bigwedge_{p\in \mathcal{T}_{j,l}}\left(\left(x_p,y_p,z_p\right)\notin \mathcal{S}\left({\mathbf{d}},r_{\mathrm{max}}\right)\right),
\end{equation}
\begin{equation}\label{secondcontainmentcond}
\forall \mathbf{r}\in \partial \mathcal{S},\qquad     \bigwedge_{j\in \mathcal{P}}\bigwedge_{l=1}^{m_j}\bigwedge_{\mathcal{T}_{j,l}=\left\{v_1,\cdots,v_4\right\}}\left(\mathbf{\Omega}_3\left(\mathbf{p}_{v_1},\mathbf{p}_{v_2},\mathbf{p}_{v_3},\mathbf{p}_{v_4},\mathbf{r}\right)\not\ge\mathbf{0}\right),
\end{equation}
\end{subequations}
where $\partial \mathcal{S}\left({\mathbf{d}},r_{\mathrm{max}}\right)$  is the boundary of the containment ball. In Eq. \eqref{firstcontainmentcond}, $p\in \mathcal{T}_{j,l}$ is the index number of {\color{black}one of the nodes of} tetrahedron $\mathcal{T}_{j,l}$ {\color{black}that is} positioned at $\left(x_p,y_p,z_p\right)$ for $j\in \mathcal{P}$ and $l=1,\cdots,m_j$. In Eq. \eqref{secondcontainmentcond},  $\mathbf{p}_{v_1}$, $\mathbf{p}_{v_2}$, $\mathbf{p}_{v_3}$, and $\mathbf{p}_{v_4}$ denote positions of vertices $v_1$, $v_2$, $v_3$, and $v_4$ of tetrahedron $\mathcal{T}_{j,l}$ for $j\in \mathcal{P}$ and $l=1,\cdots,m_j$. 
\end{definition}
The constraint  equation \eqref{firstcontainmentcond} ensures that vertices of the containment polytopes are all outside the ball $\mathcal{S}$. Also, condition \eqref{secondcontainmentcond} requires that the center of the containment ball is outside of all polytopes defined by $\mathcal{P}$.

\begin{remark}
The safety condition \eqref{firstcontainmentcond} is necessary but not sufficient for ensuring of the MQS collision avoidance with obstacles. Fig. \ref{obstacleavoidance} illustrates a situation in which collision is not avoided because the safety condition \eqref{secondcontainmentcond} is violated while \eqref{firstcontainmentcond} is satisfied.  More specifically, Fig. \ref{obstacleavoidance} shows that vertices of a tetrahedron enclosing an obstacle are outside of containment ball $\mathcal{S}$, where $\mathcal{S}$ contains the MQS. However, the containment ball enclosing the MQS is contained by the tetrahedron representing obstacle in the motion space.
% Therefore, safety condition \eqref{firstcontainmentcond} is satisfied. However, Fig. \ref{obstacleavoidance} shows that the tetrahedron encloses. As a result, the safety condition \eqref{secondcontainmentcond} is not satisfied and the leading triangle with vertices occupied by the MQS leaders collides the obstacle.
\end{remark}
\begin{figure}[htb]
\centering
\includegraphics[width=3.3  in]{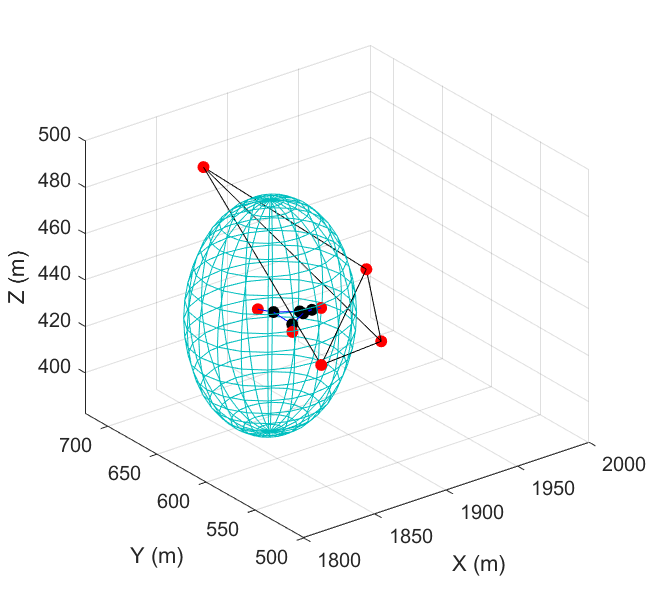}
\caption{Violation of collision avoidance requirements: MQS leaders are contained by the containment ball while the tetrahedron, representing an obstacle,  encloses the containment ball in the motion space. }
\label{obstacleavoidance}
\end{figure}

\subsubsection{A* Optimizer Functionality}
\label{A* Optimizer Functionality}
To plan the desired coordination of the MQS, we represent the coordination space by a finite number of nodes obtained by uniform discretization of the motion space. Let $\mathcal{D}_x=\left\{\Delta x,2\Delta x,\cdots,n_x\Delta x\right\}$, $\mathcal{D}_y=\left\{\Delta y,2\Delta y,\cdots,n_y\Delta x\right\}$, and $\mathcal{D}_z=\left\{\Delta z,2\Delta z,\cdots,n_z\Delta z\right\}$ define all possible discrete values for the $x$, $y$, and $z$ components of the nodes distributed in the motion space. Then, 
\begin{equation}
\resizebox{0.99\hsize}{!}{%
$
    \mathcal{D}=\left\{\tilde{\mathbf{d}}=\left(\tilde{d}_x\Delta x,\tilde{d}_y \Delta y,\tilde{d}_z \Delta z\right):\tilde{d}_x\Delta x\in \mathcal{D}_x,\tilde{d}_y\Delta y\in \mathcal{D}_y,\tilde{d}_z\Delta z\in \mathcal{D}_z\right\}
$
}
\end{equation}
defines positions of the nodes in the motion space. 
\begin{assumption}
The containment polytopes {\color{black}enclosing obstacles} are defined such that $\mathcal{P}\subset \mathcal{D}$. 
% , i.e. vertices of the polytopes enclosing obstacles are co-incident on the 
% This implies that the vertices of the containment polytopes are all coincided on the node
\end{assumption}

\begin{definition}
We define
\begin{equation}
\begin{split}
    \mathcal{F}=&\bigg\{\tilde{\mathbf{d}}\in \mathcal{D}: \left(\bigwedge_{j\in \mathcal{P}}\bigwedge_{l=1}^{m_j}\bigwedge_{p\in \mathcal{T}_{j,l}}\left(\left(x_p,y_p,z_p\right)\notin \mathcal{S}\left(\tilde{\mathbf{d}},r_{\mathrm{max}}\right)\right)\right)\wedge\\
    &\left(\bigwedge_{j\in \mathcal{P}}\bigwedge_{l=1}^{m_j}\bigwedge_{\mathcal{T}_{j,l}=\left\{v_1,\cdots,v_4\right\}}\left(\mathbf{\Omega}_3\left(\mathbf{p}_{v_1},\mathbf{p}_{v_2},\mathbf{p}_{v_3},\mathbf{p}_{v_4},\mathbf{r}\right)\not\ge\mathbf{0}\right)\right),~\\
    &\mathrm{for~}\mathbf{r}\in\partial \mathcal{S}\left(\tilde{\mathbf{d}},r_{\mathrm{max}}\right)\bigg\}\subset \mathcal{D}
\end{split}
\end{equation}
as the set of valid positions for the center of ball $\mathcal{S}$.
\end{definition}
\begin{assumption}
Initial and final positions of the containment ball are defined such that $\bar{\mathbf{d}}_s\in \mathcal{F}$ and $\bar{\mathbf{d}}_{n_\tau}\in \mathcal{F}$.
% This implies that the vertices of the containment polytopes are all coincided on the node
\end{assumption}
\begin{definition}
Set
\begin{equation}
\resizebox{0.99\hsize}{!}{%
$
    \mathcal{A}\left(\tilde{\mathbf{d}}\right)=\left\{\left(\tilde{\mathbf{d}}+\left(h_x\Delta_x,h_y\Delta_y,h_z\Delta_z\right)\right)\in \mathcal{F}:h_x,h_y,h_z\in \{-1,0,1\}\right\}
$
}
\end{equation}
defines all possible valid neighboring points of point $\tilde{\mathbf{d}}\in \mathcal{F}$.
\end{definition}
\begin{definition}
 For every $\tilde{\mathbf{d}}\in \mathcal{F}$, the straight line distance
 \begin{equation}
     C_H\left(\tilde{\mathbf{d}},\bar{\mathbf{d}}_u\right)=\|\tilde{\mathbf{d}}-\bar{\mathbf{d}}_{u}\|
 \end{equation}
 is considered as the heuristic {\color{black}cost} of position vector $\tilde{\mathbf{d}}\in \mathcal{F}$.
\end{definition}
\begin{definition}
 For every $\tilde{\mathbf{d}}\in \mathcal{F}$ and $\tilde{\mathbf{d}}'\in \mathcal{A}\left(\tilde{\mathbf{d}}\right)$,  \begin{equation}
     C_O\left(\tilde{\mathbf{d}},\tilde{\mathbf{d}}'\right)=\|\tilde{\mathbf{d}}-\tilde{\mathbf{d}}'\|
 \end{equation}
 is the operation cost for the movement from $\tilde{\mathbf{d}}\in \mathcal{F}$ towards $\tilde{\mathbf{d}}'\in \mathcal{A}\left(\tilde{\mathbf{d}}\right)$.
\end{definition}

\begin{algorithm}\label{astaralgorithmmm}
  \caption{A* Planning of the MQS Coordination}\label{euclid1}
  \begin{algorithmic}[1]
    % \Procedure{Euclid}{$a,b$}\Comment{The g.c.d. of a and b}
    %   \State \textbf{Initialization}
      \State \textit{Get: $\bar{\mathbf{d}}_s$ and $\bar{\mathbf{d}}_{u}$}
      \State \textit{Define:} Open set $\mathcal{O}=\left\{\bar{\mathbf{d}}_s\right\}$, Closed set $\mathcal{C}=\emptyset$, and $\tilde{\mathbf{d}}_{\mathrm{best}}=\bar{\mathbf{d}}_s$
      \While{$\tilde{\mathbf{d}}_{\mathrm{best}}=\bar{\mathbf{d}}_u$ or $\mathcal{O}\neq \emptyset$}
    %   \Comment{We have the answer if r is 0}
        \State  $\tilde{\mathbf{d}}_{\mathrm{best}}\leftarrow \argmin\limits_{\tilde{\mathbf{d}}\in \mathcal{F}}\left(g\left(\tilde{\mathbf{d}}\right)+C_H\left(\tilde{\mathbf{d}},\bar{\mathbf{d}}_u\right)\right)$
        \State Update $\mathcal{O}$: $\mathcal{O}\leftarrow \mathcal{O}\setminus \left\{\tilde{\mathbf{d}}_{\mathrm{best}}\right\}$
        \State Update $\mathcal{C}$: $\mathcal{C}\leftarrow \mathcal{C}\bigcup \left\{\tilde{\mathbf{d}}_{\mathrm{best}}\right\}$
        \State Assign $\mathcal{A}\left(\tilde{\mathbf{d}}_{\mathrm{best}}\right)$
        \State $\mathcal{R}\left(\tilde{\mathbf{d}}_{\mathrm{best}}\right)\leftarrow \mathcal{A}\left(\tilde{\mathbf{d}}_{\mathrm{best}}\right)\setminus \left(\mathcal{A}\left(\tilde{\mathbf{d}}_{\mathrm{best}}\right)\bigcap \mathcal{C}\right)$
        \For{\texttt{< every $\tilde{\mathbf{d}}\in \mathcal{R}\left(\tilde{\mathbf{d}}_{\mathrm{best}}\right)$>}}
        \State $\tilde{\mathbf{b}}\left(\tilde{\mathbf{d}}\right)\leftarrow \tilde{\mathbf{d}}$
         \If{$\tilde{\mathbf{d}}\in \mathcal{O}$}           
             \If{$g\left(\tilde{\mathbf{d}}_{\mathrm{best}}\right)+C_O\left(\tilde{\mathbf{d}}_{\mathrm{best}},\tilde{\mathbf{d}}\right)<g\left(\tilde{\mathbf{d}}\right)$} 
             \State $g\left(\tilde{\mathbf{d}}\right)\leftarrow g\left(\tilde{\mathbf{d}}_{\mathrm{best}}\right)+C_O\left(\tilde{\mathbf{d}}_{\mathrm{best}},\tilde{\mathbf{d}}\right)$
             \State $\tilde{\mathbf{b}}\left(\tilde{\mathbf{d}}\right)\leftarrow \tilde{\mathbf{d}}_{\mathrm{best}}$
             \EndIf
         \EndIf
                         %  \State \texttt{<do stuff>}
   \EndFor
   \State $\mathcal{O}\leftarrow \mathcal{R}\left(\tilde{\mathbf{d}}_{\mathrm{best}}\right)\bigcup \mathcal{O}$
   \EndWhile
    % \EndProcedure
  \end{algorithmic}
\end{algorithm}

 Given initial and final locations of the center of the containment ball $\mathcal{S}$, denoted by $\bar{\mathbf{d}}_s$ and $\bar{\mathbf{d}}_u$, the A* search algorithm is applied to determine optimal intermediate positions $\bar{\mathbf{b}}_s$, $\cdots$, $\bar{\mathbf{b}}_{m_\tau}$ along the optimal path of the containment ball $\mathcal{S}$ from $\bar{\mathbf{d}}_s$ to $\bar{\mathbf{d}}_u$ in an obstacle-laden environment (See Algorithm \ref{euclid1}). More specifically, the A* optimizer generates $\bar{\mathbf{b}}_s$, $\cdots$, $\bar{\mathbf{b}}_{m_\tau}$ by searching over set $\mathcal{F}$, where
 \begin{subequations}
 \begin{equation}
     \bar{\mathbf{b}}_s=\bar{\mathbf{d}}_s,
 \end{equation}
 \begin{equation}
     \bar{\mathbf{b}}_{m_\tau}=\bar{\mathbf{d}}_u,
 \end{equation}
  \begin{equation}
     \left(\bar{\mathbf{b}}_{k},\bar{\mathbf{b}}_{k+1}\right)\in \mathcal{A}\left(\bar{\mathbf{b}}_{k}\right).
 \end{equation}
 \end{subequations}
The center of the containment ball $\mathcal{S}$ moves along the straight paths obtained by connecting $\bar{\mathbf{b}}_s$, $\cdots$, $\bar{\mathbf{b}}_{m_\tau}$. Therefore, $n_\tau $ serially-connected line segments defines the optimal path of the containment ball, where $n_\tau\leq m_\tau$, $\bar{\mathbf{d}}_s=\bar{\mathbf{b}}_s$, $\bar{\mathbf{d}}_{n_\tau}=\bar{\mathbf{b}}_{m_\tau}=\bar{\mathbf{d}}_u$,  and the end point of the $k$-th line segment connects  $\bar{\mathbf{d}}_{k-1}$ to $\bar{\mathbf{d}}_{k}$. Given $\bar{\mathbf{b}}_s$, $\cdots$, $\bar{\mathbf{b}}_{m_\tau}$, algorithm \ref{euclid} is used to determine $\bar{\mathbf{d}}_1$, $\cdots$, $\bar{\mathbf{d}}_{n_\tau-1}$.

\begin{algorithm}\label{astaralgorithmmm}
  \caption{Assignment of Optimal Way-points $\bar{\mathbf{d}}_1$, $\cdots$, $\bar{\mathbf{d}}_{n_\tau-1}$}\label{euclid}
  \begin{algorithmic}[1]
    % \Procedure{Euclid}{$a,b$}\Comment{The g.c.d. of a and b}
    %   \State \textbf{Initialization}
      \State \textit{Get: $\bar{\mathbf{b}}_s=\bar{\mathbf{d}}_s$, $\cdots$, $\bar{\mathbf{b}}_{m_\tau}=\bar{\mathbf{d}}_u$}
      \State \textit{Set: $i=0$}
      \For{\texttt{< $k\leftarrow1$ to $m_\tau-1$ >}}
        % \State $\tilde{\mathbf{b}}\left(\tilde{\mathbf{d}}\right)\leftarrow \tilde{\mathbf{d}}$
         \If{$\bar{\mathbf{b}}_k-\bar{\mathbf{b}}_{k-1}\neq\bar{\mathbf{b}}_{k+1}-\bar{\mathbf{b}}_{k}$}  
         \State $i\leftarrow i+1$
  \State $\bar{\mathbf{d}}_i=\bar{\mathbf{b}}_k$
             \EndIf
                         %  \State \texttt{<do stuff>}
   \EndFor
    % \EndProcedure
  \end{algorithmic}
\end{algorithm}

% The desired final and initial configuration of the leaders can be related by 
% \begin{equation}
%     \mathbf{y}_{L,h,n_\tau}=\mathbf{D}\left(\mathbf{I}_3\otimes \mathbf{Q}_{n_\tau}\right)
% \end{equation}
\subsection{Intermediate Configuration of the Leading Triangle}
\label{Planning2}
% In this section, we first offer an approach for decomposition of matrix $\mathbf{Q}(t)$ by defining $\mathbf{Q}(t)$ based on rotation angle $\theta_r(t)$, a shear deformation angle $\theta_d(t)$, and the principal eigenvalues of the continuum deformation coordination, denoted by $\sigma_1(t)$ and $\sigma_2(t)$. Then,  
Matrix $\mathbf{Q}_{xy}(t)$ can be expressed by
\begin{equation}\label{Qxy}
    \mathbf{Q}_{xy}(t)=\mathbf{R}_{xy}(t)\mathbf{U}_{xy}(t),
\end{equation}
where rotation matrix $\mathbf{R}_{xy}(t)$ and pure deformation matrix $\mathbf{U}_{xy}(t)$ are defined as follows:
\begin{subequations}
\begin{equation}\label{Rxy}
    \mathbf{R}_{xy}(t)=\begin{bmatrix}\cos \theta_r&-\sin \theta_r\\
    \sin\theta_r&\cos \theta_r\end{bmatrix},
\end{equation}
\begin{equation}\label{Uxy}
    \mathbf{U}_{xy}(t)=\mathbf{R}_D(t)\mathbf{\Lambda}(t)\mathbf{R}_D^T(t),
    % \begin{bmatrix}\cos \theta_d&-\sin \theta_d\\
    % \sin\theta_d&\cos \theta_d\end{bmatrix}\begin{bmatrix}\sigma_1&0\\
    % 0&\sigma_2\end{bmatrix}\begin{bmatrix}\cos \theta_d&-\sin \theta_d\\
    % \sin\theta_d&\cos \theta_d\end{bmatrix}
\end{equation}
\end{subequations}
{\color{black}where}
\begin{subequations}
\begin{equation}\label{LAMBDAAAAAAA}
    \mathbf{\Lambda}(t)=\begin{bmatrix}
    \sigma_1(t)&0\\
    0&\sigma_2(t)
    \end{bmatrix}
    ,
\end{equation}
\begin{equation}\label{RDDDDDDDDD}
    \mathbf{R}_D(t)=\begin{bmatrix}\cos \theta_d&-\sin \theta_d\\
    \sin\theta_d&\cos \theta_d\end{bmatrix}.
\end{equation}
\end{subequations}
Note that $\theta_r(t)>0$ and $\theta_d(t)>0$ are the rotation and shear deformation angles; and $\sigma_1(t)$ and $\sigma_2(t)$ are the first and second deformation eigenvalues. 
% Per constraint equation \eqref{c1}, the area of the leading triangle remains constant at any time $t$. 
Because $\mathbf{\Lambda}(t)$ is positive definite and diagonal, matrix  $\mathbf{U}_{xy}(t)$ is positive definite at any time $t\in \left[t_s,t_u\right]$ \cite{rastgoftar2020scalable}. 
\begin{proposition}\label{prop1}
Matrix $\mathbf{U}_{xy}^m$ can be expressed as
\begin{equation}\label{Uxym}
    \mathbf{U}_{xy}^m(t)=\begin{bmatrix}a_m(t)&b_m(t)\\b_m(t)&a_m(t)\end{bmatrix},
\end{equation}
with
\begin{subequations}
\begin{equation}\label{amm}
    a_m(t)=\sigma_1^m(t)\cos^2\theta_d(t)+\sigma_2^m(t)\sin^2\theta_d(t),
\end{equation}
\begin{equation}\label{bmm}
    b_m(t)=\left(\sigma_1^m(t)-\sigma_2^m(t)\right)\sin\theta_d(t)\cos\theta_d(t),
\end{equation}
\begin{equation}\label{cmm}
    c_m(t)=\sigma_1^m(t)\sin^2\theta_d(t)+\sigma_2^m(t)\cos^2\theta_d(t).
\end{equation}
\end{subequations}
Also, $\sigma_1$, $\sigma_2$, and $\theta_d$ can be related to $a_m$, $b_m$, and $c_m$ by
\begin{subequations}
\begin{equation}\label{lambda1}
    \sigma_1(t)=\sqrt[m]{\dfrac{a_m(t)+c_m(t)}{2}+\sqrt{\big[{1\over 2}\left(a_m(t)-c_m(t)\right)\big]^2+b_m^2(t)}},
\end{equation}
\begin{equation}\label{lambda2}
    \sigma_2(t)=\sqrt[m]{\dfrac{a_m(t)+c_m(t)}{2}-\sqrt{\big[{1\over 2}\left(a_m(t)-c_m(t)\right)\big]^2+b_m^2(t)}},
\end{equation}
\begin{equation}\label{thetad}
        \theta_d(t)=\dfrac{1}{2}\tan^{-1}\left(\dfrac{2b_m(t)}{a_m(t)-c_m(t)}\right).
\end{equation}
\end{subequations}
\end{proposition}

\begin{proof}
Because $\mathbf{R}_{D}(t)$ is orthogonal at time $t$, $\mathbf{R}_{D}^T(t)\mathbf{R}_{D}(t)=\mathbf{I}_2$. If matrix $\mathbf{U}_{xy}^m$ {\color{black}is} expressed as 
\begin{equation}\label{UDANYM}
    \mathbf{U}_{xy}^m(t)=\mathbf{R}_{D}(t)\mathbf{\Lambda}^m\mathbf{R}_{D}(t),
\end{equation}
for $m=1,2,\cdots$, then,
\begin{equation}\label{UDM}
    \begin{split}
    \mathbf{U}_{xy}^{m+1}(t)=&\mathbf{R}_{D}(t)\mathbf{\Lambda}\mathbf{R}_{D}^T(t)\mathbf{R}_{D}(t)\mathbf{\Lambda}^m\mathbf{R}_{D}(t)\\
    =&\mathbf{R}_{D}(t)\mathbf{\Lambda}^{m+1}\mathbf{R}_{D}(t).
\end{split}
\end{equation}
{\color{black}Since} Eq. \eqref{UDANYM} is valid for $m=0${\color{black}, Eq.} \eqref{UDM} ensures that Eq. \eqref{UDANYM} is valid for any $m>0$. By replacing \eqref{LAMBDAAAAAAA} and \eqref{RDDDDDDDDD} into \eqref{UDANYM}, elements of matrix $\mathbf{U}_{xy}^m$ ($a_m$, $b_m$, $c_m$) are obtained by Eqs. \eqref{amm}, \eqref{bmm}, and \eqref{cmm}.
\end{proof}

% \begin{theorem}
% Eigenvalues $\sigma_1(t)$ and $\sigma_2(t)$ and shear deformation angle $\theta_d(t)$ can be expressed with respect to the leaders' desired $\mathbf{y}_{L,HT}(t)$ at any time $t$ by
% \begin{subequations}
% \begin{equation}\label{lambda1}
%     \sigma_1(t)=\sqrt[m]{\dfrac{a_m+c_m}{2}+\sqrt{\big[{1\over 2}\left(a_m-c_m\right)\big]^2+b^2}}\
% \end{equation}
% \begin{equation}\label{lambda2}
%     \sigma_2(t)=\sqrt[m]{\dfrac{a_m+c_m}{2}2\sqrt{\big[{1\over 2}\left(a_m-c_m\right)\big]^2+b^2}}
% \end{equation}
% \begin{equation}\label{thetad}
%         \theta_d(t)=\dfrac{1}{2}\tan^{-1}\left(\dfrac{2b_m}{a_m-c_m}\right).
% \end{equation}
% \end{subequations}
% \end{theorem}
By provoking Proposition \ref{prop1}, matrix $\mathbf{U}_{xy}^2=\mathbf{Q}_{xy}^T\mathbf{Q}_{xy}$ \cite{rastgoftar2020scalable} can be expressed in the form of Eq. \eqref{Uxym} where $m=2$ and 
\begin{subequations}
\begin{equation}
    a_2(t)=\mathbf{y}_{L,HT}^T(t)\mathbf{O}^T\mathbf{\Gamma}^T\left(\mathbf{E}_1^T\mathbf{E}_1+\mathbf{E}_3^T\mathbf{E}_3\right)\mathbf{\Gamma}\mathbf{O}\mathbf{y}_{L,HT}(t),
\end{equation}
\begin{equation}
    b_2(t)=\mathbf{y}_{L,HT}^T(t)\mathbf{O}^T\mathbf{\Gamma}^T\left(\mathbf{E}_1^T\mathbf{E}_2+\mathbf{E}_3^T\mathbf{E}_4\right)\mathbf{\Gamma}\mathbf{O}\mathbf{y}_{L,HT}(t),
\end{equation}
\begin{equation}
    c_2(t)=\mathbf{y}_{L,HT}^T(t)\mathbf{O}^T\mathbf{\Gamma}^T\left(\mathbf{E}_2^T\mathbf{E}_2+\mathbf{E}_4^T\mathbf{E}_4\right)\mathbf{\Gamma}\mathbf{O}\mathbf{y}_{L,HT}(t).
\end{equation}
\end{subequations}
Therefore, we can determine $\sigma_1(t)$, $\sigma_2(t)$, and {\color{black}$\theta_d(t)$} by replacing $m=2$, $a_m(t)=a_2(t)$, $b_m(t)=b_2(t)$, and $c_m(t)=c_2(t)$ into Eqs. \eqref{lambda1}, \eqref{lambda2}, and \eqref{thetad} at time $t\in \left[t_s,t_u\right]$. Furthermore, matrix $\mathbf{R}_{xy}(t)=\mathbf{Q}\mathbf{U}_{xy}^{-1}$ is related to $\mathbf{y}_{L,HT}(t)$ by
\begin{equation}\label{RXYPOS}
    \mathbf{R}_{xy}(t)=\begin{bmatrix}
    \mathbf{E}_1\mathbf{\Gamma}\mathbf{O}\mathbf{y}_{L,HT}(t)&\mathbf{E}_2\mathbf{\Gamma}\mathbf{O}\mathbf{y}_{L,HT}(t)\\
    \mathbf{E}_3\mathbf{\Gamma}\mathbf{O}\mathbf{y}_{L,HT}(t)&\mathbf{E}_4\mathbf{\Gamma}\mathbf{O}\mathbf{y}_{L,HT}(t)\\
    \end{bmatrix}
    \begin{bmatrix}
    a_2(t)&b_2(t)\\
    b_2(t)&c_2(t)
    \end{bmatrix}
    ^{-1\over 2}.
\end{equation}
{\color{black}Therefore,} rotation angle $\theta_r(t)$ is obtained at any time $t\in \left[t_s,t_u\right]$ {\color{black}by knowing rotation matrix $\mathbf{R}_{xy}(t)$ over time interval $\left[t_s,t_u\right]$}.
\begin{proposition}\label{prop2}
If the area of the leading triangle remains constant at any time $t\in [t_s,t_u]$, then the following conditions hold:
\begin{subequations}
\begin{equation}\label{ac1}
    \sigma_2(t)={1\over \sigma_1(t)},
\end{equation}
\begin{equation}\label{ac2}
    a_2(t)c_2(t)-b_2{\color{black}^2}=1.
\end{equation}
\end{subequations}
\end{proposition}
\begin{proof}
Per Assumption \ref{assmq1}, $\mathbf{U}_{xy}(t_s)=\mathbf{I}_2$. If the area of the leading triangle remains constant, then $\sigma_{1}(t)\sigma_2(t)=\sigma_{1}(t_s)\sigma_2(t_s)=1$ and $\left|\mathbf{U}_{xy}(t)\right|{\color{black}= a_2(t)c_2(t)-b_2^2}=1$ at any time $t$. Therefore, conditions \eqref{ac1} and \eqref{ac2} hold at any time $t\in \left[t_s,t_u\right]$.
\end{proof}

\begin{theorem}
Assume every quadcopter $i\in \mathcal{V}$ can be enclosed by a ball of radius $\epsilon$, and it can execute a proper control input $\mathbf{u}_i$ such that
\begin{equation}\label{BoundedDeviation}
    \bigwedge_{i\in \mathcal{V}}\|\mathbf{r}_{i}(t)-\mathbf{p}_i(t)\|\leq \delta,\qquad \forall t\in \left[t_s,t_u\right].
\end{equation}
Let
\begin{equation}
    d_{\mathrm{min}}=\min\limits_{i,j\in \mathcal{V},~j\neq i}\left\|\mathbf{p}_{i,0}-\mathbf{p}_{j,o}\right\|,
\end{equation}
be the minimum separation distance between two quadcopters.
% ,  and
% \begin{equation}
%     d_{\mathrm{max}}=\max\limits_{i\in \mathcal{V}}\left\|\mathbf{p}_{i,0}\right\|,
% \end{equation}
% be the maximum distance of the a quadcopter from the center of containment ball in the reference configuration. 
Then, collision between every two quadcopers and collision of the MQS with obstacles are both avoided, if
% \begin{equation}
%     d_{\mathrm{min}}\geq 2\left(\delta+\epsilon\right)
% \end{equation}
% and 
the largest eigenvalue of matrix $\mathbf{U}_{xy}$ satisfies inequality constraint
\begin{equation}\label{firstsigma1}
\sigma_1(t)\leq
\dfrac{d_\mathrm{min}}{2\left(\delta+\epsilon\right)},
% \sigma_{\mathrm{max}},\qquad \forall t\in \left[t_s,t_u\right]
\end{equation}
and every quadcopter remains inside the containment ball {\color{black}$\mathcal{S}\left(\mathbf{d}\left(t\right),r_{\mathrm{max}}\right)$} at any time $t\in \left[t_s,t_u\right]$.
% where 
% \begin{equation}\label{sigmaamax}
%     \sigma_{\mathrm{max}}\leq \min\left\{\dfrac{d_s}{2\left(\delta+\epsilon\right)}, \dfrac{r_{\mathrm{max}}}{d_{\mathrm{max}}-\delta-\epsilon }\right\}.
% \end{equation}
\end{theorem}
\begin{proof}
Per Eqs. \eqref{lambda1} and \eqref{lambda2}, $\sigma_2(t)\leq \sigma_1(t)$ at any time $t\in \left[t_s,t_u\right]$. Collision between every two quadcopters is avoided, if \cite{rastgoftar2020scalable}
\begin{equation}\label{interagentcollision}
    \sigma_2(t)\geq {2\left(\delta+\epsilon\right)\over d_{\mathrm{min}}},\qquad \forall t\in \left[t_s,t_u\right].
\end{equation}
 Per Proposition \ref{prop2}, $\sigma_2(t)={1\over \sigma_1(t)}$. Thus, Eq. \eqref{interagentcollision} can be rewritten as follows:
 \begin{equation}\label{interagentcollisionmain}
    \sigma_1(t)\leq  {2\left(\delta+\epsilon\right)\over d_{\mathrm{min}}},\qquad \forall t\in \left[t_s,t_u\right].
\end{equation}
By applying A* search method, we ensure that the containment ball does not hit obstacles in the motion space. 
% , eigen-decomposition, and optimal control planning we assure that the containment ball  $\mathcal{S}(t)$ moves along a collision-free path from the initial position to the final destination. Therefore,
Therefore, obstacle collision avoidance is guaranteed{\color{black},} if  quadcopters are all inside the containment ball $\mathcal{S}\left(\mathbf{d}\left(t\right),r_{\mathrm{max}}\right)$ at any time $t\in \left[t_s,t_u\right]$.
% and the containment bal does not hit the obstacles . This containment condition is assured, if \eqref{firstsigma1} is satisfied and
% \begin{equation}\label{collisionmainobstacle}
%     \sigma_1(t)\leq \dfrac{r_{\mathrm{max}}}{d_{\mathrm{max}}-\delta-\epsilon }.
% \end{equation}
% By combining conditions \eqref{interagentcollisionmain} and \eqref{collisionmainobstacle}, Eq. \eqref{sigmaamax} assigns an upper bound for the largest egenvalues of matrix $\mathbf{U}_{xy}(t)$ over the time interval $\left[t_s,t_u\right]$.
\end{proof}
% \end{equation}

\textbf{Intermediate Configurations Leaders:}  
We offer a procedure {\color{black}with the following five main steps} to determine the intermediate waypoints of the leaders:

\textit{Step 1:} Given $\bar{\mathbf{y}}_{L,HT,n_\tau}=\mathbf{y}_{L,HT}\left(t_{u}\right)$, $\sigma_{1,n_\tau}=\sigma_1(t_u)$, $\theta_{d,n_\tau}=\theta_d(t_u)$, and $\theta_{r,n_\tau}=\theta_r(t_u)$ are computed using Eqs. \eqref{lambda1}, \eqref{thetad}, and \eqref{RXYPOS}, respectively. 

\textit{Step 2:} We compute
\begin{subequations}
\begin{equation}
    \sigma_{1,k}=\beta_k\sigma_{1,0}+(1-\beta_k)\sigma_{1,n_\tau},
\end{equation}
\begin{equation}
    \theta_{d,k}=(1-\beta_k)\theta_{d,n_\tau},
\end{equation}
\begin{equation}
    \theta_{r,k}=(1-\beta_k)\theta_{r,n_\tau}
\end{equation}
\end{subequations}
for $k=1,\cdots,n_\tau-1${\color{black}, where $\beta_k$ is computed using Eq. \eqref{betak}.}

\textit{Step 3:} We compute $\sigma_{2,k}={1\over \sigma_{1,k}}$ for $k=1,\cdots,n_\tau-1$.

\textit{Step 4:} Given $\sigma_{1,k}${\color{black},}  $\sigma_{2,k}$, and $\theta_{d,k}${\color{black},} matrix $\mathbf{U}_{xy,k}=\mathbf{U}_{xy}\left(t_k\right)$ is obtained by Eq. \eqref{Uxy} for $k=1,\cdots,n_\tau-1$. Also, matrix $\mathbf{R}_{xy,k}=\mathbf{R}_{xy}\left(t_k\right)$ is obtained using Eq. \eqref{Rxy} by knowing the rotation angle  $\theta_{r,k}$ for $k=1,\cdots,n_\tau-1$.

\textit{Step {\color{black}5}:} By knowing  $\mathbf{R}_{xy,k}=\mathbf{R}_{xy}\left(t_k\right)$  and $\mathbf{U}_{xy,k}=\mathbf{U}_{xy}\left(t_k\right)${\color{black},} the Jacobian matrix  $\mathbf{Q}_{xy,k}=\mathbf{Q}_{xy}\left(t_k\right)$ is obtained using Eq. \eqref{Qxy}. Then, we can use relation \eqref{globaldesiredcoordination} to obtain $\bar{\mathbf{y}}_{L,HT,k}$ by replacing $\mathbf{Q}_{xy,k}=\mathbf{Q}_{xy}\left(t_k\right)$ and $\bar{\mathbf{d}}_k$ for $k=1,\cdots,n_\tau-1$.

\subsection{Optimal Control Planning}
\label{Planning3}
{\color{black}This section offers an optimal control solution to determine the leaders' desired trajectories  connecting every two consecutive waypoints $\bar{\mathbf{y}}_{L,HT,k}$ and $\bar{\mathbf{y}}_{L,HT,k+1}$ for $k=0,1,\cdots,n_\tau-1$, where $z$ components of the leaders is defined by Eq. \eqref{zcomponent}, and $x$ and $y$ components the leaders' desired trajectories are governed by \eqref{maindynamicsss}.}

\textbf{Coordination Constraint:} Per equality constraint \eqref{c1}, the area of the leading triangle, given by
\begin{equation}\label{aa}
    A(t)=\mathbf{y}_{L,HT}^T(t)\mathbf{\Psi}\mathbf{y}_{L,HT}(t),
\end{equation}
must be equal to constant value $A_s$ at any time $t\in \left[t_s,t_u\right]$. This equality constraint is satisfied, if $\mathbf{y}_{L,HT}(t)$ is updated by dynamics \eqref{maindynamicsss},  $c\left(\mathbf{x}_L,\mathbf{u}_L\right)=\ddot{A}\left(t\right)=0$ at any time $t\in \left[t_k,t_{k+1}\right]$ for $k=0,1,\cdots,n_\tau-1$, and the following boundary conditions are satisfied:
\begin{subequations}
\begin{equation}\label{c11}
\resizebox{0.99\hsize}{!}{%
$
k=0,1,\cdots,n_\tau,\qquad {\mathbf{y}}_{L,HT}^T\left(t_k\right)\mathbf{D}^T\mathbf{O}^T\mathbf{P}^T\mathbf{P}\mathbf{O}\mathbf{D}\mathbf{y}_{L,HT}\left(t_k\right)-A_s=0,
$
}
\end{equation}
\begin{equation}\label{c22}
\resizebox{0.99\hsize}{!}{%
$
k=0,1,\cdots,n_\tau,\qquad \dot{\mathbf{y}}_{L,HT}^T\left(t_k\right)\mathbf{D}^T\mathbf{O}^T\mathbf{P}^T\mathbf{P}\mathbf{O}\mathbf{D}\mathbf{y}_{L,HT}\left(t_k\right)=0.
$
}
\end{equation}
\end{subequations}
% and $c\left(\mathbf{x}_L,\mathbf{u}_L\right)=\ddot{A}\left(t\right)=0$ at any time $t\in \left[t_k,t_{k+1}\right]$ for $k=0,1,\cdots,n_\tau-1$. 
By taking the second time derivative of ${A}\left(t\right)$, $c\left(\mathbf{x}_L,\mathbf{u}_L\right)$ is obtained as follows:
\begin{equation}\label{mainequality}
    c\left(\mathbf{x}_L,\mathbf{u}_L,t\right)=
    \mathbf{x}_L^T\mathbf{\Gamma}_{\mathbf{xx}}\mathbf{x}_L+2\mathbf{x}_L^T\mathbf{\Gamma}_{\mathbf{xu}}\mathbf{u}_L=0,
    % \begin{bmatrix}
    % \mathbf{x}_L\left(t\right)\\
    % \mathbf{u}_L\left(t\right)\\
    % \end{bmatrix}^T\mathbf{\Gamma}(t)\begin{bmatrix}
    % \mathbf{x}_L\left(t\right)\\
    % \mathbf{u}_L\left(t\right)\\
    % \end{bmatrix}
\end{equation}
where 
\begin{subequations}
\begin{equation}\label{gxx}
    \mathbf{\Gamma}_{\mathbf{xx}}=2\begin{bmatrix}
    \mathbf{0}_{6\times 6}&\mathbf{P}\\
    \mathbf{P}&\mathbf{0}_{6\times 6}\\
    \end{bmatrix}
    ,
\end{equation}
\begin{equation}\label{gxu}
    \mathbf{\Gamma}_{\mathbf{xu}}=\begin{bmatrix}
    \mathbf{P}\\
    \mathbf{0}_{6\times 6}\\
    \end{bmatrix}
    .
\end{equation}
\end{subequations}

The objective of the optimal control planning is to determine the desired trajectories of the leaders by minimization of cost function 
\begin{equation}\label{costttttttttttttttttt}
k=0,1,\cdots,n_\tau-1,\qquad \mathrm{J}={1\over 2}\int_{t_k\left(t_u\right)}^{t_{k+1}\left(t_u\right)}\mathbf{u}_L^T(t)\mathbf{u}_L^T(t)dt
\end{equation}
subject to boundary conditions
\begin{subequations}\label{boundaryconditionsss}
\begin{equation}
    \mathbf{x}_L(t_k)=\bar{\mathbf{x}}_{L,k},
\end{equation}
\begin{equation}\label{conditiontkplus1}
    \mathbf{x}_L(t_{k+1})=\bar{\mathbf{x}}_{L,k+1},
\end{equation}
\end{subequations}
and equality constraint \eqref{mainequality} at any time $t\in \left[t_k\left(t_u\right),t_{k+1}\left(t_u\right)\right]$ for $k=0,1,\cdots,n_\tau-1$ where $t_k(t_u)$ is obtained by \eqref{tku}. 
\begin{theorem}
Suppose leaders' desired trajectories are updated by dynamics \eqref{maindynamicsss} such that equality constraint \eqref{mainequality} is satisfied at any time $t\in \left[t_k\left(t_u\right),t_{k+1}\left(t_u\right)\right]$ given the boundary conditions in Eq. \eqref{boundaryconditionsss}. Assuming the ultimate time $t_u$ is given, $t_k$ and $t_{k+1}$ obtained by Eq. \eqref{tku} are fixed, and the optimal desired trajectories of leaders minimizing the cost function \eqref{costttttttttttttttttt} are governed by dynamics
\begin{equation}\label{rawwwwwwwwwwwwwwwwa}
    \begin{bmatrix}
    \dot{\mathbf{x}}_L\\
    \dot{\lambda}
    \end{bmatrix}
    =
    \mathbf{A}_{\mathbf{x\lambda}}\left(\gamma(t)\right)
    % \mathbf{A}_{\mathrm{CD}}
    \begin{bmatrix}
    {\mathbf{x}}_L\\
    {\lambda}
    \end{bmatrix}
   ,
    % +\begin{bmatrix}
    % \mathbf{v}_{\mathrm{CD},1}\\
    % \mathbf{v}_{\mathrm{CD},2}
    % \end{bmatrix}
\end{equation}
where 
\begin{subequations}
\begin{equation}\label{62aagamma}
    \mathbf{A}_{\mathbf{x\lambda}}\left(\gamma(t)\right)=\begin{bmatrix}
    \mathbf{A}_L-2\gamma(t)\mathbf{B}_L\mathbf{\Gamma}_{\mathbf{xu}}^T&-\mathbf{B}_L\mathbf{B}_L^T\\
   -2\gamma \mathbf{\Gamma}_{\mathbf{xx}}+4\gamma^2(t) \mathbf{\Gamma}_{\mathbf{xu}}\mathbf{\Gamma}_{\mathbf{xu}}^T&-\mathbf{A}_L^T+2\gamma(t)\mathbf{\Gamma}_{\mathbf{xu}}\mathbf{B}_L^T\\
    \end{bmatrix}
    ,
\end{equation}\label{gammmaammamamama}
\begin{equation}
       \gamma\left(t\right)=\dfrac{\mathbf{x}_L^T\mathbf{\Gamma}_{\mathbf{xx}}\mathbf{x}_L+\mathbf{\Gamma}_{\mathbf{x}}^T\mathbf{x}_L-2\mathbf{x}_L^T\mathbf{\Gamma}_{\mathbf{xu}}\mathbf{B}_L^T\mathbf{\lambda}}{4\mathbf{x}_L^T\mathbf{\Gamma}_{\mathbf{xu}}\mathbf{\Gamma}_{\mathbf{xu}}^T\mathbf{x}_L},
\end{equation}
\end{subequations}
and $\lambda \in \mathbb{R}^{18\times 1}$ is the co-state vector. In addition,
% \begin{equation}
%      \mathbf{\Phi}=\begin{bmatrix}
%      \mathbf{\Phi}_{11}\left(t,t_k\right)&\mathbf{\Phi}_{12}\left(t,t_k\right)\\
%      \mathbf{\Phi}_{21}\left(t,t_k\right)&\mathbf{\Phi}_{22}\left(t,t_k\right)\\
%      \end{bmatrix}=\mathrm{exp}\left(\int_{t_k}^t\mathbf{A}_{\mathbf{x\lambda}}\left(\gamma(s)\right)ds\right).
% \end{equation}
the state vector $\mathbf{x}_{L}(t)$ and co-state vector  $\lambda(t)$ are obtained by
\begin{subequations}
\begin{equation}\label{xlt}
\begin{split}
    \mathbf{x}_L(t)=&\left(\mathbf{\Phi}_{11}\left(t,t_k\right)-\mathbf{\Phi}_{12}\left(t,t_{k+1}\right)\mathbf{\Phi}_{11}\left(t_{k+1},t_k\right)\right)\bar{\mathbf{x}}_{L,k}\\
    +&\mathbf{\Phi}_{12}\left(t,t_{k+1}\right)\bar{\mathbf{x}}_{L,k+1}
\end{split}
\end{equation}
\begin{equation}\label{lambdalt}
\begin{split}
    \lambda(t)=&\left(\mathbf{\Phi}_{21}\left(t,t_k\right)-\mathbf{\Phi}_{22}\left(t,t_{k+1}\right)\mathbf{\Phi}_{11}\left(t_{k+1},t_k\right)\right)\bar{\mathbf{x}}_{L,k}\\
    +&\mathbf{\Phi}_{22}\left(t,t_{k+1}\right)\bar{\mathbf{x}}_{L,k+1}
\end{split}
\end{equation}
\end{subequations}
at time $t\in \left[t_k,t_{k+1}\right]$, where
\begin{equation}\label{SatateTransitionMatrrrixxxxxxxs}
     \mathbf{\Phi}=\begin{bmatrix}
     \mathbf{\Phi}_{11}\left(t,t_k\right)&\mathbf{\Phi}_{12}\left(t,t_k\right)\\
     \mathbf{\Phi}_{21}\left(t,t_k\right)&\mathbf{\Phi}_{22}\left(t,t_k\right)\\
     \end{bmatrix}=\mathrm{exp}\left(\int_{t_k}^t\mathbf{A}_{\mathbf{x\lambda}}\left(\gamma(s)\right)ds\right)
\end{equation}
is the state transition matrix with partitions $\mathbf{\Phi}_{11}\left(t,t_k\right)\in \mathbb{R}^{12\times12}$, $\mathbf{\Phi}_{12}\left(t,t_k\right)\in \mathbb{R}^{12\times12}$, $\mathbf{\Phi}_{21}\left(t,t_k\right)\in \mathbb{R}^{12\times12}$, and $\mathbf{\Phi}_{22}\left(t,t_k\right)\in \mathbb{R}^{12\times12}$.
% solving
% \begin{equation}\label{rawwwwwwwwwwwwwwwwa}
%     \begin{bmatrix}
%     \dot{\mathbf{x}}_L\\
%     \dot{\lambda}
%     \end{bmatrix}
%     =
%     \mathbf{A}_{\mathbf{x\lambda}}
%     % \mathbf{A}_{\mathrm{CD}}
%     \begin{bmatrix}
%     {\mathbf{x}}_L\\
%     {\lambda}
%     \end{bmatrix}
%   ,
%     % +\begin{bmatrix}
%     % \mathbf{v}_{\mathrm{CD},1}\\
%     % \mathbf{v}_{\mathrm{CD},2}
%     % \end{bmatrix}
% \end{equation}
% where $\lambda$ is the co-state vector,
% \begin{equation}
%     \mathbf{A}_{\mathbf{x\lambda}}\left(\gamma(t)\right)=\begin{bmatrix}
%     \mathbf{A}_L-2\gamma\mathbf{B}_L\mathbf{\Gamma}_{\mathbf{xu}}^T&-\mathbf{B}_L\mathbf{B}_L^T\\
%   -2\gamma \mathbf{\Gamma}_{\mathbf{xx}}+4\gamma^2 \mathbf{\Gamma}_{\mathbf{xu}}\mathbf{\Gamma}_{\mathbf{xu}}^T&-\mathbf{A}_L^T+2\gamma\mathbf{\Gamma}_{\mathbf{xu}}\mathbf{B}_L^T\\
%     \end{bmatrix}
%     ,
% \end{equation}
% and
% \begin{equation}
%     \gamma\left(t\right)=\dfrac{\mathbf{x}_L^T\mathbf{\Gamma}_{\mathbf{xx}}\mathbf{x}_L+\mathbf{\Gamma}_{\mathbf{x}}^T\mathbf{x}_L-2\mathbf{x}_L^T\mathbf{\Gamma}_{\mathbf{xu}}\mathbf{B}_L^T\mathbf{\lambda}}{4\mathbf{x}_L^T\mathbf{\Gamma}_{\mathbf{xu}}\mathbf{\Gamma}_{\mathbf{xu}}^T\mathbf{x}_L}.
% \end{equation}
% Also, 
% \begin{equation}
%     \begin{bmatrix}
%     {\mathbf{x}}_L(t)\\
%     {\lambda}(t)
%     \end{bmatrix}
% \end{equation}
\end{theorem}

% Per Theorem \eqref{theoremcosnta}, constraint \eqref{mainequality} ensures that the area of the leading triangle remains constant at any time $t\in \left[t_k,t_{k+1}\right]$ for $k=0,1,\cdots,n_\tau-1$.
\begin{proof}
The optimal leaders' trajectories are determined by minimization of the augmented cost function 
\begin{equation}\label{JP}
\resizebox{0.99\hsize}{!}{%
$
    \mathrm{J}_a=\int_{t_k}^{t_{k+1}}\left({1\over 2}\mathbf{u}_L^T\mathbf{u}_L^T+\mathbf{\lambda}^T\left(\mathbf{A}_L{\mathbf{x}}_L+\mathbf{B}_L{\mathbf{u}}_L-\dot{\mathbf{x}}_L\right)+\gamma c\left(\mathbf{x}_L,\mathbf{u}_L\right)\right)dt,
    $
    }
\end{equation}
where $\mathbf{\lambda}\in \mathbb{R}^{12\times 1}$ is the co-state vector and $\gamma(t)$ is the Lagrange multiplier. By taking variation from the augmented cost function \eqref{JP}, we can write
\begin{equation}\label{JPP}
\begin{split}
    \delta \mathrm{J}_a=\int_{t_k}^{t_{k+1}}&\Bigg[\delta \mathbf{u}_L^T\left(\mathbf{u}_L+\mathbf{B}_L^T\mathbf{\lambda}+\gamma {\partial c\over \partial \mathbf{u}_L}\right)+\delta \mathbf{x}_L^T\left(\dot{\mathbf{\lambda}}+\mathbf{A}_L^T\mathbf{\lambda}+\gamma {\partial c\over \partial \mathbf{x}_L}\right)\\
    +&\delta \mathbf{\lambda}^T\left(\mathbf{A}_L{\mathbf{x}}_L+\mathbf{B}_L{\mathbf{u}}_L-\dot{\mathbf{x}}_L\right)\bigg]dt=0,
\end{split}
  \end{equation}
%   Now, we can substitute ${\partial c\over \partial \mathbf{x}_L}\in \mathbb{R}^{27\times 1}$ and ${\partial c\over \partial \mathbf{u}_L}\in \mathbb{R}^{9\times 1}$ by
where ${\partial c\over \partial \mathbf{x}_L}=2\mathbf{\Gamma}_{\mathbf{xx}}\mathbf{x}_L+2\mathbf{\Gamma}_{\mathbf{xu}}\mathbf{u}_L$ and ${\partial c\over \partial \mathbf{u}_L}=2\mathbf{\Gamma}_{\mathbf{xu}}^T\mathbf{x}_L$.
%  \begin{subequations}
%  \begin{equation}
%      {\partial c\over \partial \mathbf{x}_L}=2\mathbf{\Gamma}_{\mathbf{xx}}\mathbf{x}_L+2\mathbf{\Gamma}_{\mathbf{xu}}\mathbf{u}_L,
%  \end{equation}
%   \begin{equation}
%      {\partial c\over \partial \mathbf{u}_L}=2\mathbf{\Gamma}_{\mathbf{xu}}^T\mathbf{x}_L,
%  \end{equation}
%  \end{subequations}
% are obtained by taking partial derivatives from constraint Eq. \eqref{mainequality}, where $\mathbf{\Gamma}_{\mathbf{xx}}$ and $\mathbf{\Gamma}_{\mathbf{xu}}$ are defined in \eqref{gxx} and \eqref{gxu}, respectively. 
By imposing $\delta\mathrm{J}_a=0$, the state dynamics \eqref{maindynamicsss} is obtained, the co-state dynamics become
  \begin{equation}
      \dot{\mathbf{\lambda}}=-\mathbf{A}_L^T\mathbf{\lambda}-\gamma(t)\dfrac{\partial c}{\partial \mathbf{x}_L},
  \end{equation}
  and $\mathbf{u}_L$ is obtained as follows:
  \begin{equation}\label{ul}
\mathbf{u}_{L}=-\mathbf{B}_L^T\mathbf{\lambda}-\gamma {\partial c\over \partial \mathbf{u}_L}=-\mathbf{B}_L^T\mathbf{\lambda}-2\gamma(t) \mathbf{\Gamma}_{\mathbf{xu}}^T\mathbf{x}_L.
\end{equation}
%  \begin{subequations}
%  \begin{equation}
% \mathbf{C}_x=\begin{bmatrix}
% \mathbf{I}_{27}&\mathbf{0}_{27\times 9}
% \end{bmatrix}
% \mathbf{\Gamma}\left(t\right)\begin{bmatrix}
% \mathbf{I}_{27}\\\mathbf{0}_{9\times 27}
% \end{bmatrix}
% ,
%  \end{equation}
%   \begin{equation}
%   \mathbf{C}_u=\begin{bmatrix}
% \mathbf{0}_{9\times 27}&\mathbf{I}_{9}
% \end{bmatrix}
% \mathbf{\Gamma}\left(t\right)\begin{bmatrix}
% \mathbf{0}_{27\times 9}\\\mathbf{I}_{9}
% \end{bmatrix}
% \begin{bmatrix}
% \mathbf{0}_{9\times 18}&\mathbf{I}_{9}
% \end{bmatrix}
% .
%  \end{equation}
%  \end{subequations}
By substituting $\mathbf{u}_L=-\mathbf{B}_L^T\mathbf{\lambda}-2\gamma(t) \mathbf{\Gamma}_{\mathbf{xu}}\mathbf{x}_L$, the equality constraint \eqref{mainequality} is converted to
\begin{equation}
\label{organizedequality}
    c\left(\mathbf{x}_L,\mathbf{u}_L\right)=\left(4\mathbf{x}_L^T\mathbf{\Gamma}_{\mathbf{xu}}\mathbf{\Gamma}_{\mathbf{xu}}^T\mathbf{x}_L\right)\gamma(t)+\mathbf{x}_L^T\mathbf{\Gamma}_{\mathbf{xx}}\mathbf{x}_L-2\mathbf{x}_L^T\mathbf{\Gamma}_{\mathbf{xu}}\mathbf{B}_L^T\mathbf{\lambda}=0.
\end{equation}
% where
% \begin{subequations}
% \begin{equation}
%     c_1\left(\mathbf{x}_L,\mathbf{u}_L,t\right)=4\mathbf{x}_L^T\mathbf{\Gamma}_{\mathbf{xu}}\mathbf{\Gamma}_{\mathbf{xu}}^T\mathbf{x}_L
%     % -\left(2\mathbf{x}_L^T\mathbf{\Gamma}_{\mathbf{xu}}+\mathbf{\Gamma}_{\mathbf{u}}^T\right)\left(2\mathbf{\Gamma}_{\mathbf{xu}}^T\mathbf{x}_L+\mathbf{\Gamma}_{\mathbf{u}}\right)
% \end{equation}
% \begin{equation}
%     c_2\left(\mathbf{x}_L,\mathbf{u}_L,t\right)=\mathbf{x}_L^T\mathbf{\Gamma}_{\mathbf{xx}}\mathbf{x}_L-2\mathbf{x}_L^T\mathbf{\Gamma}_{\mathbf{xu}}\mathbf{B}_L^T\mathbf{\lambda}.
%     % -\left(2\mathbf{x}_L^T\mathbf{\Gamma}_{\mathbf{xu}}+\mathbf{\Gamma}_{\mathbf{u}}^T\right)\mathbf{B}_L^T\lambda +\mathbf{x}_L^T\mathbf{\Gamma}_{\mathbf{xx}}\mathbf{x}_L+\mathbf{\Gamma}_{\mathbf{x}}\mathbf{x}_L+\Gamma_{00}.
% \end{equation}
% \end{subequations}
% \begin{subequations}
% \begin{equation}
%     c_1\left(\mathbf{x}_L,\mathbf{u}_L,t\right)=4\mathbf{x}_L^T\mathbf{\Gamma}_{\mathbf{xu}}\mathbf{\Gamma}_{\mathbf{xu}}^T\mathbf{x}_L+6\mathbf{x}_L^T\mathbf{\Gamma}_{\mathbf{xu}}^T\mathbf{\Gamma}_{\mathbf{u}}+2\mathbf{\Gamma}_{\mathbf{u}}^T\mathbf{\Gamma}_{\mathbf{u}},
% \end{equation}
% \begin{equation}
%     c_2\left(\mathbf{x}_L,\mathbf{u}_L,t\right)=\left(2\mathbf{x}_L^T\mathbf{\Gamma}_{\mathbf{xu}}+\mathbf{\Gamma}_{\mathbf{u}}^T\right)\mathbf{B}_L^T\mathbf{\lambda}-\mathbf{x}_L^T\mathbf{\Gamma}_{\mathbf{xx}}\mathbf{x}_L-\mathbf{\Gamma}_{\mathbf{x}}^T\mathbf{x}_L-\Gamma_{00}.
% \end{equation}
% \end{subequations}
By substituting $\mathbf{u}_L=-\mathbf{B}_L^T\mathbf{\lambda}-2\gamma(t) \mathbf{\Gamma}_{\mathbf{xu}}\mathbf{x}_L-\gamma\mathbf{\mathbf{\Gamma}}_u$ into Eq. \eqref{maindynamicsss}, we also obtain the leaders' desired trajectories solving dynamics \eqref{rawwwwwwwwwwwwwwwwa}. The solution of dynamics \eqref{rawwwwwwwwwwwwwwwwa} is given by
\begin{equation}\label{optimalsolutiuo}
 \begin{bmatrix}
    {\mathbf{x}}_L(t)\\
    {\lambda}(t)
    \end{bmatrix}
    =
    \begin{bmatrix}
     \mathbf{\Phi}_{11}\left(t,t_k\right)&\mathbf{\Phi}_{12}\left(t,t_k\right)\\
     \mathbf{\Phi}_{21}\left(t,t_k\right)&\mathbf{\Phi}_{22}\left(t,t_k\right)\\
     \end{bmatrix}
    %  \mathbf{\Phi}_{\mathbf{x\lambda}}\left(t,t_k\right)
    % \mathbf{A}_{\mathrm{CD}}
    \begin{bmatrix}
    \bar{\mathbf{x}}_{L,k}\\
    {\lambda}_{k}
    \end{bmatrix}
\end{equation}
at time $t\in [t_k,t_{k+1}]$, where ${\lambda}_{k}=\lambda\left(t_k\right)$.
% , where
% \begin{equation}
%     \mathbf{\Phi}_{\mathbf{x\lambda}}\left(t,t_k\right)=\mathrm{exp}\left(\int_{t_k}^t\mathbf{A}_{\mathbf{x\lambda}}\left(\gamma(s)\right)ds\right), 
% \end{equation}
% is the state transition matrix that is partitioned as follows:
% \begin{equation}
%      \mathbf{\Phi}_{\mathbf{x\lambda}}\left(t,t_k\right)=\begin{bmatrix}
%      \mathbf{\Phi}_{11}\left(t,t_k\right)&\mathbf{\Phi}_{12}\left(t,t_k\right)\\
%      \mathbf{\Phi}_{21}\left(t,t_k\right)&\mathbf{\Phi}_{22}\left(t,t_k\right)\\
%      \end{bmatrix}
% \end{equation}
By imposition  boundary condition \eqref{conditiontkplus1}, 
\begin{equation}
    \lambda_k=\mathbf{\Phi}_{12}\left(t_k,t_{k+1}\right)\left(\mathbf{x}_L\left(t_{k+1}\right)-\mathbf{\Phi}_{11}\left(t_{k+1},t_k\right)\mathbf{x}_L\left(t_{k}\right)\right)
\end{equation}
is obtained from Eq. \eqref{optimalsolutiuo}. By substituting $\lambda_k$ into Eq. \eqref{optimalsolutiuo}, $\mathbf{x}_L(t)$ is obtained by Eq. \eqref{xlt} at any time $t\in \left[t_k,t_{k+1}\right]$.

\end{proof}

% is obtained by \eqref{organizedequality}. Now, we partition 
% \begin{}

% where 
% \begin{subequation}
% \begin{equation}
%     \mathbf{A}_{\mathrm{CD},1}(t)=\mathbf{A}_L-2\gamma\mathbf{B}_L\mathbf{\Gamma}_{\mathbf{xu}}^T,
% \end{equation}
% \begin{equation}
%     \mathbf{A}_{\mathrm{CD},2}(t)=-\mathbf{B}_L\mathbf{B}_L^T,
% \end{equation}
% \begin{equation}
%     \mathbf{A}_{\mathrm{CD},3}(t)=-2\gamma \mathbf{\Gamma}_{\mathbf{xx}}+4\gamma^2 \mathbf{\Gamma}_{\mathbf{xu}}\mathbf{\Gamma}_{\mathbf{xu}}^T,
% \end{equation}
% \begin{equation}
%     \mathbf{A}_{\mathrm{CD},4}(t)=-\mathbf{A}_L^T+2\gamma\mathbf{\Gamma}_{\mathbf{xu}}\mathbf{B}_L^T,
% \end{equation}
% \begin{equation}
%     \mathbf{v}_{\mathrm{CD},1}(t)=-\gamma(t)\mathbf{B}_L\mathbf{\Gamma}_u,
% \end{equation}
% \begin{equation}
%     \mathbf{v}_{\mathrm{CD},2}(t)=\gamma^2(t)\mathbf{\Gamma}_{\mathbf{xu}}\mathbf{\Gamma}_u-\gamma(t)\mathbf{\Gamma}_x,
% \end{equation}
% \end{subequation}
% \begin{subequations}
% \begin{equation}
% \mathbf{A}_{\mathrm{CD}}(t)=
%     \begin{bmatrix}
%     &-\mathbf{B}_L\mathbf{B}_L^T\\
%     -2\gamma \mathbf{\Gamma}_{\mathbf{xx}}+4\gamma^2 \mathbf{\Gamma}_{\mathbf{xu}}\mathbf{\Gamma}_{\mathbf{xu}}^T&-\mathbf{A}_L^T+2\gamma\mathbf{\Gamma}_{\mathbf{xu}}\mathbf{B}_L^T
%     \end{bmatrix}
% \end{equation}
% \begin{equation}
% \mathbf{v}_{\mathrm{CD}}=
%     \begin{bmatrix}
%     -\gamma(t)\mathbf{B}_L\mathbf{\Gamma}_u\\
%     \gamma^2(t)\mathbf{\Gamma}_{\mathbf{xu}}\mathbf{\Gamma}_u-\gamma(t)\mathbf{\Gamma}_x
%     \end{bmatrix}
%     ,
% \end{equation}
% \end{subequations}

\begin{algorithm}
  \caption{Assignment of travel time $t_u$ and desired trajectory $\mathbf{x}_L(t)$ over $\left[t_0,t_u\right]$}\label{euclid33}
  \begin{algorithmic}[1]
    % \Procedure{Euclid}{$a,b$}\Comment{The g.c.d. of a and b}
    %   \State \textbf{Initialization}
      \State \textit{Get:} $\bar{\mathbf{x}}_{L,0}$, $\cdots$, $\bar{\mathbf{x}}_{L,n_\tau}$ and $\beta_0$, $\cdots$, $\beta_{n_\tau-1}$, $\epsilon_T$, $\epsilon_\gamma$, small $T_{\mathrm{min}}$ and large $T_{\mathrm{max}}$ ($t_{u,\mathrm{min}}< t_u< t_{u,\mathrm{max}}$) 
      \State  \textit{Set:} small $T_{\mathrm{min}}$, large $T_{\mathrm{max}}$, $t_0=0$, $t_1=0$, $\cdots$, $t_{n_\tau-1}=0$
     \State $t_u={T_{\mathrm{min}}+T_{\mathrm{max}}\over 2}$
      \While{$t_u-T_{\mathrm{min}}\geq \epsilon_T$}
      \For{\texttt{< $k\leftarrow 0$ to $n_\tau-1$}}
      \State $t_k\leftarrow \beta_k t_u$
      \State $t_{k+1}\leftarrow \beta_{k+1} t_u$
      \State $\gamma'(t)=0$ at every time $t\in \left[t_k,t_{k+1}\right]$
      \State $\gamma(t)=0$ at every time $t\in \left[t_k,t_{k+1}\right]$
      \State $e_\gamma\leftarrow 2\epsilon \gamma$ 
       \While{$e_\gamma\geq \epsilon_\gamma$}
       \State Compute $\mathbf{A}{\color{black}_{\mathbf{x}\lambda}}\left(\gamma(t)\right)$ using Eq. \eqref{62aagamma}
       \State Compute $\mathbf{\Phi}\left(t,t_k\right)$ $\mathbf{x}_L\left(t\right)$ using Eq. \eqref{SatateTransitionMatrrrixxxxxxxs}
       \State Obtain $\mathbf{x}_L\left(t\right)$ by Eq. \eqref{xlt} for $t\in\left[t_k,t_{k+1}\right]$
       \State Obtain $\lambda\left(t\right)$ by Eq. \eqref{lambdalt} for $t\in\left[t_k,t_{k+1}\right]$
       \State Compute $\gamma'(t)$ for $t\in\left[t_k,t_{k+1}\right]$:
       \State $\gamma'\left(t\right)=\dfrac{\mathbf{x}_L^T\mathbf{\Gamma}_{\mathbf{xx}}\mathbf{x}_L+\mathbf{\Gamma}_{\mathbf{x}}^T\mathbf{x}_L-2\mathbf{x}_L^T\mathbf{\Gamma}_{\mathbf{xu}}\mathbf{B}_L^T\mathbf{\lambda}}{4\mathbf{x}_L^T\mathbf{\Gamma}_{\mathbf{xu}}\mathbf{\Gamma}_{\mathbf{xu}}^T\mathbf{x}_L}$
       \State $e_\gamma=\max\limits_{t\in \left[t_k,t_{k+1}\right]}~\left|\gamma(t)-\gamma'(t)\right|$
       \State $\gamma(t)=\gamma'(t)$
       \EndWhile
       \EndFor
       \State $e_T=\max\limits_{t\in \left[t_0,t_u\right]}~\bigwedge_{i\in \mathcal{V}}\|\mathbf{r}_{i}(t)-\mathbf{p}_i(t)\|$
    %   \If {$P_{s}\: is\: maximum $} 
    %         \State
    % %  \ElseIf {$P_{E}$ is maximum}
        \If{$e_T\leq \delta$}
        \State $T_{\mathrm{max}}\leftarrow t_u$
        \EndIf
        \If{$e_T> \delta$}
        % \State \textbf{else}
        \State $T_{\mathrm{min}}\leftarrow t_u$
        \EndIf
%       \State
% %              \If{$g\left(\tilde{\mathbf{d}}_{\mathrm{best}}\right)+C_O\left(\tilde{\mathbf{d}}_{\mathrm{best}},\tilde{\mathbf{d}}\right)<g\left(\tilde{\mathbf{d}}\right)$} 
% %              \State $g\left(\tilde{\mathbf{d}}\right)\leftarrow g\left(\tilde{\mathbf{d}}_{\mathrm{best}}\right)+C_O\left(\tilde{\mathbf{d}}_{\mathrm{best}},\tilde{\mathbf{d}}\right)$
% %              \State $\tilde{\mathbf{b}}\left(\tilde{\mathbf{d}}\right)\leftarrow \tilde{\mathbf{d}}_{\mathrm{best}}$
%               \EndIf
       \EndWhile
  \end{algorithmic}
\end{algorithm}

% \section{Quadcopter Dynamics}\label{Quadcopter Dynamics}
\section{Continuum Deformation Acquisition}
\label{Continuum Deformation Acquisition}
This paper considers collective motion of a quadcopter team consisting of $N$ quadcopters, where dynamics of quadcopter $i\in \mathcal{V}$ is given by
{
\begin{equation}
\label{generalnonlineardynamics}
\begin{cases}
    \dot{\mathbf{x}}_i=\mathbf{f}_i\left(\mathbf{x}_i\right)+\mathbf{g}_i\left(\mathbf{x}_i\right)\mathbf{u}_i\\
    \mathbf{r}_i=\mathbf{C}\mathbf{x}_i
    \end{cases}
    .
\end{equation}
}
In \eqref{generalnonlineardynamics},
$
    \mathbf{x}_i=
    \begin{bmatrix}
    \mathbf{r}_i^T&\dot{\mathbf{r}}_i^T&\phi_i&\theta_i&\psi_i&{\bf{\omega}}_i^T
    \end{bmatrix}
    ^T
$
is the state,  $\mathbf{u}_i=
    \begin{bmatrix}
    p_i&\tau_{\phi,i}&\tau_{\theta,i}&\tau_{\psi,i}
    \end{bmatrix}
    ^T$ is the input, {$\mathbf{C}_i=\begin{bmatrix}
    \mathbf{I}_3&\mathbf{0}_{3\times 9}
    \end{bmatrix}$,}
   \[
    {
   \resizebox{0.99\hsize}{!}{%
$
\mathbf{f}_i\left(\mathbf{x}_i\right)=
\begin{bmatrix}
\dot{\mathbf{r}}_i\\
{1\over m_i}p_i\hat{\mathbf{k}}_{b,i}-g\hat{\mathbf{e}}_3\\
\mathbf{\Gamma}_i^{-1}\left(\phi_i,\theta_i,\psi_i\right){\bf{\omega}}_i\\
\mathbf{J}_i^{-1}{\bf{\omega}}_i\times \left(\mathbf{J}_i{\bf{\omega}}_i\right)\\
\end{bmatrix},
~\mathrm{and}~\mathbf{g}_i\left(\mathbf{x}_i\right)=
\begin{bmatrix}
\mathbf{0}_{3\times 1}&\mathbf{0}_{3\times 3}\\
{1\over m_i}\hat{\mathbf{k}}_{b,i}&\mathbf{0}_{3\times 3}\\
\mathbf{0}_{3\times 1}&\mathbf{0}_{3\times 1}\\
\mathbf{0}_{3\times 1}&\mathbf{J}_i^{-1}\\
\end{bmatrix}
,
$
}}
\]
where $m_i$ and $\mathbf{J}_i$ are  the mass and mass moment of inertia of quadcopter $i\in \mathcal{V}$, respectively, $\mathbf{0}_{3\times 1}\in \mathbb{R}^{3\times 1}$, $\mathbf{0}_{3\times 3}\in \mathbb{R}^{3\times {3}}${, and $\mathbf{0}_{3\times 9}\in \mathbb{R}^{3\times {9}}$} are the zero-entry matrices, $\mathbf{I}_3\in\mathbb{R}^{3\times 3}$ is the identity matrix, $g=9.81m/s^2$ is the gravity, and 
\begin{equation}
\label{Eq55}
\mathbf{\Gamma}_i\left(\phi_i,\theta_i,\psi_i\right)=
    \begin{bmatrix}
    1&0&-\sin\theta_i\\
    0&\cos\phi_i&\cos\theta_i\sin\phi_i\\
    0&-\sin\phi_i&\cos\phi_i\cos\theta_i
    \end{bmatrix}
    .
\end{equation}
The dynamics of leader and follower quadcopter sub-teams are given by
\begin{subequations}
\begin{equation}
\label{Leaders}
    \begin{cases}
    \dot{\mathbf{x}}_L=\mathbf{F}_L\left(\mathbf{x}_L\right)+\mathbf{G}_L\left(\mathbf{x}_L\right)\mathbf{u}_L\\
    \mathbf{y}_L=\mathbf{C}_L
    % \left(\mathbf{I}_3\otimes\mathbf{C}\right)
    \mathbf{x}_L 
    \end{cases}
    ,
\end{equation}
\begin{equation}
\label{Followers}
    \begin{cases}
    \dot{\mathbf{x}}_F=\mathbf{F}_F\left(\mathbf{x}_F\right)+\mathbf{G}_F\left(\mathbf{x}_F\right)\mathbf{u}_L\\
    \mathbf{y}_F=\mathbf{C}_F
    % \left(\mathbf{I}_{N-3}\otimes\mathbf{C}\right)
    \mathbf{x}_F 
    \end{cases}
    ,
\end{equation}
\end{subequations}
where $\mathbf{C}_L\in \mathbb{R}^{9\times 36}$, $\mathbf{C}_F\in \mathbb{R}^{3\left(N-3\right)\times 12\left(N-3\right)}$, $\mathbf{x}_L=\begin{bmatrix}
\mathbf{x}_1^T&\cdots&\mathbf{x}_3^T
\end{bmatrix}^T$ and $\mathbf{x}_F=\begin{bmatrix}
\mathbf{x}_4^T&\cdots&\mathbf{x}_N^T
\end{bmatrix}^T$ are the state vectors of leaders and followers, $\mathbf{u}_L=\begin{bmatrix}
\mathbf{u}_1^T&\cdots&\mathbf{u}_3^T
\end{bmatrix}^T$ and $\mathbf{u}_F=\begin{bmatrix}
\mathbf{u}_4^T&\cdots&\mathbf{u}_N^T
\end{bmatrix}^T$ are the input vectors of leaders and followers, $\mathbf{y}_L=\begin{bmatrix}
\mathbf{r}_1^T&\cdots&\mathbf{r}_3^T
\end{bmatrix}^T$ and $\mathbf{y}_F=\begin{bmatrix}
\mathbf{r}_4^T&\cdots&\mathbf{r}_N^T
\end{bmatrix}^T$
 are the  output vectors of leaders and followers, and $\mathbf{F}_L\left(\mathbf{x}_L\right)=\begin{bmatrix}
\mathbf{f}_1^T\left(\mathbf{x}_1\right)&\cdots&\mathbf{f}_3^T\left(\mathbf{x}_3\right)
\end{bmatrix}^T$, $\mathbf{F}_F\left(\mathbf{x}_F\right)=\begin{bmatrix}
\mathbf{f}_4^T\left(\mathbf{x}_4\right)&\cdots&\mathbf{f}_N^T\left(\mathbf{x}_B\right)
\end{bmatrix}^T$, 
$\mathbf{G}_L\left(\mathbf{x}_L\right)=\begin{bmatrix}
\mathbf{f}_1^T\left(\mathbf{x}_1\right)&\cdots&\mathbf{f}_3^T\left(\mathbf{x}_3\right)
\end{bmatrix}^T$, $\mathbf{G}_F\left(\mathbf{x}_F\right)=\begin{bmatrix}
\mathbf{f}_4^T\left(\mathbf{x}_4\right)&\cdots&\mathbf{f}_N^T\left(\mathbf{x}_B\right)
\end{bmatrix}^T$ are smooth functions.

The continuum deformation, defined by \eqref{globaldesiredcoordination} and planned by leaders $1$, $2$, and $3$, are acquired by followers in a decentralized fashion through local communication \cite{rastgoftar2020scalable}. Communication among the quadcopters are defined by graph $\mathcal{G}\left(\mathcal{V},\mathcal{E}\right)$ with the properties presented in Section \ref{Graph Theory Notions}. Here, we review the existing communication-based guidance protocol and the trajectory control design \cite{rastgoftar2020scalable} in Sections \ref{Communication-Based Guidance Protocol} and \ref{Trajectory Control Design} below.

\subsection{Communication-Based Guidance Protocol}\label{Communication-Based Guidance Protocol}
Given followers' communication weights, we define matrix
\[
    \mathbf{W}=\begin{bmatrix}
    \mathbf{0}_{3\times 3}&\mathbf{0}_{3\times \left(N-3\right)}\\
    \mathbf{B}_{\mathrm{MQS}}&\mathbf{A}_{\mathrm{MQS}}
    \end{bmatrix}\in \mathbb{R}^{\left(N-3\right)\times N}
\]
 with partitions $\mathbf{B}_{\mathrm{MQS}}\in \mathbb{R}^{\left(N-3\right)\times 3}$ and $\mathbf{A}_{\mathrm{MQS}}\in \mathbb{R}^{\left(N-3\right)\times \left(N-3\right)}$, and $(i,j)$ entry \cite{rastgoftar2020scalable}
 \begin{equation}
     W_{ij}=\begin{cases}
     w_{i,j}&i\in \mathcal{V}_F,~j\in \mathcal{N}_{i}\\
     -1&j=i\\
     0&\mathrm{otherwise}
     \end{cases}
     .
 \end{equation}
 In Ref. \cite{rastgoftar2020scalable}, we show that 
 \[
 \mathbf{y}_{HT}=\mathrm{vec}\left(\begin{bmatrix}\mathbf{p}_1(t)&\cdots&\mathbf{p}_N(t)\end{bmatrix}^T\right)\in \mathbb{R}^{3N\times 1},
 \]
 aggregating $x$, $y$, and $z$ components of global desired positions of all quadcopters, can be defined based on $\mathbf{y}_{L,HT}(t)$ by
 \begin{equation}
     \mathbf{y}_{HT}(t)=\left(\mathbf{I}_3\otimes\mathbf{W}_L\right)\mathbf{y}_{L,HT}(t),
 \end{equation}
 where 
 \begin{equation}
     \mathbf{W}_L=\begin{bmatrix}
     \mathbf{\Omega}_2^T\left(\mathbf{p}_{1,0},\mathbf{p}_{2,0},\mathbf{p}_{3,0},\mathbf{p}_{1,0}\right)\\
     \vdots\\
     \mathbf{\Omega}_2^T\left(\mathbf{p}_{1,0},\mathbf{p}_{2,0},\mathbf{p}_{3,0},\mathbf{p}_{N,0}\right)\\
     \end{bmatrix}
     \in \mathbb{R}^{N\times 3}
 \end{equation}
is defined based on $\mathbf{W}$ by
\begin{equation}
    \mathbf{W}_L=\left(-\mathbf{I}_N+\mathbf{W}\right)^{-1}\begin{bmatrix}
    \mathbf{I}_{3}&\mathbf{0}_{3\times \left(N-3\right)}
    \end{bmatrix}
    ^T.
\end{equation}
Given the output vectors of the leaders' dynamics \eqref{Leaders}, denoted by $\mathbf{y}_L$, and followers' dynamics \eqref{Followers}, denoted by $\mathbf{y}_F$, we define the MQS output vector 
  \[
 \mathbf{y}(t)=\mathbf{R}_L\mathbf{y}_L(t)+\mathbf{R}_F\mathbf{y}_F(t)
 \]
 to measure deviation of the MQS from the desired continuum deformation coordination by checking constraint \eqref{BoundedDeviation}, where $\mathbf{R}_L=\left[R_{L_{ij}}\right]\in \mathbb{R}^{3N\times 9}$ and $\mathbf{R}_F=\left[R_{F_{ij}}\right]\in \mathbb{R}^{3N\times 3(N-3)}$ are defined as follows:
 \begin{subequations}
 \begin{equation}
      R_{L_{ij}}=\begin{cases}
        1&i=j,~ j\leq3\\
        1&i=j+N,~ 4\leq j\leq6\\
        1&i=j+N,~ 7\leq j\leq9\\
        0&\mathrm{otherwise}
     \end{cases}
     ,
 \end{equation}
  \begin{equation}
      R_{F_{ij}}=\begin{cases}
        1&4\leq i\leq N,~ j\leq3\\
        1&N+4\leq i\leq 2N,~ 4<j\leq6\\
        1&2N+4\leq i\leq 3N,~ 4<j\leq6\\
        0&\mathrm{otherwise}
     \end{cases}
     .
 \end{equation}
 \end{subequations}
As shown in Fig. \ref{blockdiagram}, $\mathbf{y}_{L,HT}(t)$ is the reference input of the control system of leader coordination, and 
\begin{equation}
    \mathbf{y}_{F,d}(t)=\left(\mathbf{I}_3\otimes \mathbf{A}_{\mathrm{MQS}}\right)\mathbf{y}_F(t)+\left(\mathbf{I}_3\otimes \mathbf{B}_{\mathrm{MQS}}\right)\mathbf{y}_L(t)
\end{equation}
is the reference input of the control system of the follower quadcopter team.

\subsection{Trajectory Control Design}\label{Trajectory Control Design}
The objective of control design is to determine $\mathbf{u}_L\in \mathbb{R}^{12\times 1}$ and $\mathbf{u}_F$ such that \eqref{BoundedDeviation} is satisfied at any time $t\in \left[t_0,t_u\right]$.  We can rewrite the safety condition \eqref{BoundedDeviation} as
 \begin{equation}\label{convertedbounded}
     \bigwedge_{i\in \mathcal{V}}\left(\left(\mathbf{y}(t)-\mathbf{y}_{HT}(t)\right)^T\mathbf{S}_i^T\mathbf{S}_i\left(\mathbf{y}(t)-\mathbf{y}_{HT}(t)\right)\leq\delta^2\right),\qquad \forall t,
 \end{equation}
where $\mathbf{S}_i=\left[\mathbf{S}_{i_{pq}}\right]\in \mathbb{R}^{3\times 3N}$ is defined as follows:
\begin{equation}
    \mathbf{S}_{i_{pq}}=\begin{cases}
    1&\bigwedge_{i=1}^3\left(\left(p=i\right)\wedge \left(q=N(i-1)+i\right)\right)\\
    0&\mathrm{otherwise}
    \end{cases}
    .
\end{equation}
We use the feedback linearization approach presented in Ref. \cite{rastgoftar2020scalable}  to obtain  the control input vector $\mathbf{u}_i(t)$ for every quadcopter $i\in \mathcal{V}$ such that inequality constraint \eqref{convertedbounded} is satisfied. 
% The details of this feedback linearization design are  and Appendix \ref{MAS Collective Dynamics1}.
\begin{figure*}[htb]
\centering
\includegraphics[width=6 in]{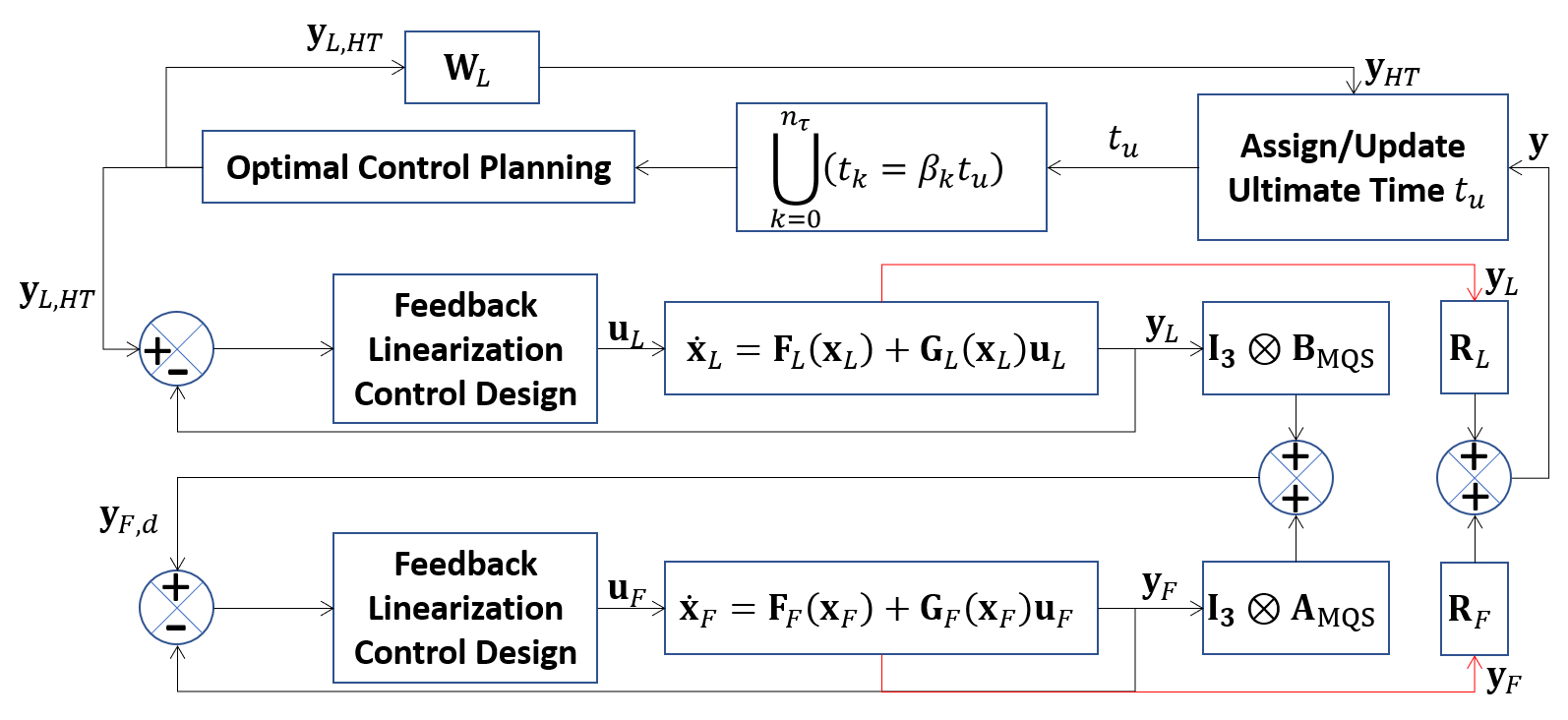}
\caption{The block diagram of the MQS continuum deformation acquisition. }
\label{blockdiagram}
\end{figure*}

\begin{figure}
 \centering
 \subfigure[]{\includegraphics[width=0.49\linewidth]{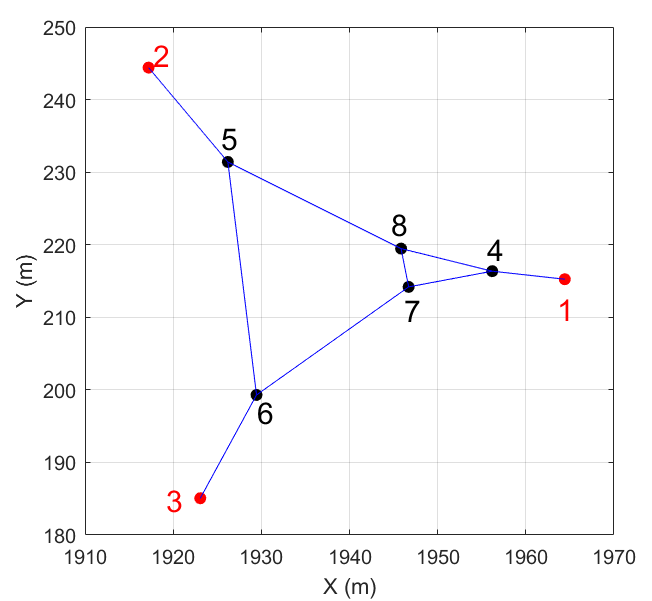}}
  \subfigure[]{\includegraphics[width=0.48\linewidth]{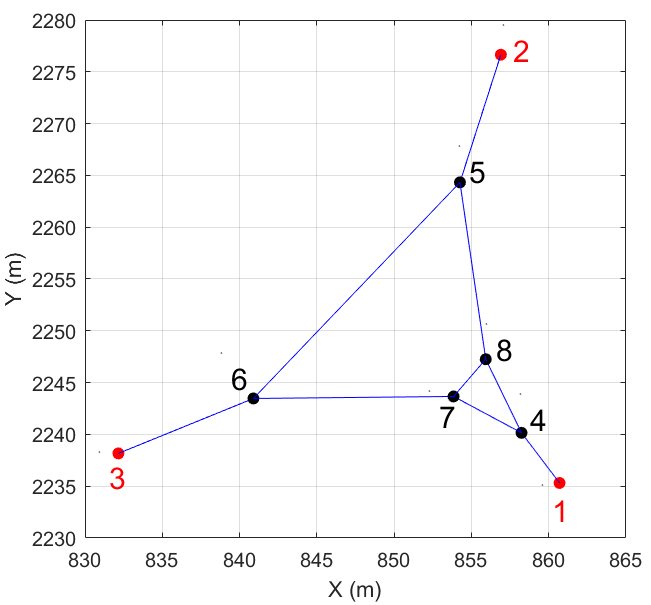}}
     \caption{(a,b) MQS initial and final formations.}
\label{Formation}
\end{figure}
\section{Simulation Results}
\label{Simulation Results}
We consider an MQS consisting of $N=8$ quadcopters with the initial formation shown in Fig. \ref{Formation} (a). The MQS is initially distributed over horizontal plane $z=43m$ where $\bar{\mathbf{d}}_s=\begin{bmatrix}
1935&215&43\end{bmatrix}^T$ is the position of the center of the containment ball $\mathcal{S}$ at the initial time $t_s=0s$. It is desired that the MQS finally reaches the final formation shown in Fig. \ref{Formation} (b) in an obstacle laden environment shown in Fig. \ref{MQS}. The final formation of the MQS is obtained by homogeneous transformation of the MQS initial formation and specified by choosing $\sigma_{1,n_\tau}=1.2$, $\sigma_{2,n_\tau}={1\over \sigma_{1,n_\tau}}=0.83$, $\theta_{d,n_\tau}=-{\pi\over 4}$, and $\bar{\mathbf{d}}_u=\begin{bmatrix}850&2250&50\end{bmatrix}^T$. 
\begin{figure}[ht]
\centering
\includegraphics[width=3.3 in]{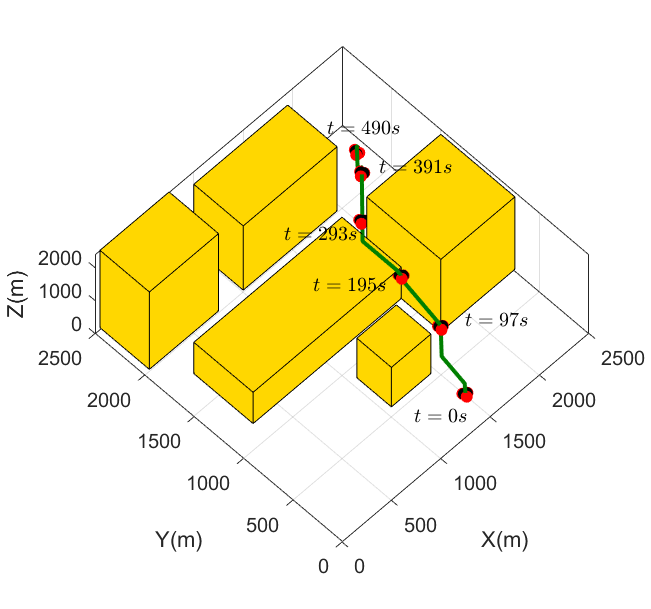}
\caption{Collective of the MQS in an obstacle-laden environment.}
\label{MQS}
\end{figure}

\textbf{Inter-agent Communication:} Given quadcopters' initial positions, followers' in-neighbors and communication weights are computed using the approach presented in Section \ref{Communication-Based Guidance Protocol} and listed in Table \ref{Table1}. Note that quadcopters' identification numbers are defined by set $\mathcal{V}=\{1,\cdots,10\}$, where $\mathcal{V}_L=\{1,2,3\}$ and $\mathcal{V}_F=\{4,\cdots,10\}$ define the identification numbers of the leader and follower quadcopters, respectively. 
\begin{table}[]
\caption{In-neighbor agents of followers $4$ through $33$ and followers' communication weights}
    \centering
    \begin{tabular}{|c|ccc|ccc|}
    \hline
    &\multicolumn{3}{c}{In-neighbors}&\multicolumn{2}{c}{Communication weights}\\
    $i\in \mathcal{V}_F$&$i_1$&$i_2$&$i_3$&$w_{i,i_1}$&$w_{i,i_2}$&$w_{i,i_3}$\\
    \hline
    4&   1&   7&   8&    0.55&    0.15    &0.30\\
    5 & 2  & 6 & 8    &0.60 &   0.15   & 0.25\\
    6  &  3  & 5  &  7   &0.60  &  0.15  &  0.25\\
    7   & 4 &  6   & 8  & 0.40   & 0.20 &   0.40\\
    8&   4&    5&   7&   0.45    &0.25&    0.30\\
%     9 &   1    &2 &   3&   -0.0539&    0.7703    &0.2836\\
%   10  &  3   &33  & 12   & 0.4701 &   0.0690   & 0.4609\\
%   11   &26  & 17   &18  & -0.2920  &  0.4454  &  0.8467\\
%   12    &1 &   2&    3 &   0.3453   & 0.3290 &   0.3257\\
%   13&   16&   31 &  32&    0.5004&    0.3318&    0.1678\\
%   14 &  23   &20  & 30    &0.4841 &   0.2493    &0.2666\\
%   15  &  7  &  4   &24   & 0.3433  &  0.3332   & 0.3235\\
%   16   & 1 &   2&    3  &  0.1792   & 0.4896  &  0.3312\\
%   17&   11&   18 &  33 &   0.5325    &0.2496 &   0.2179\\
%   18 &   1   & 2  &  3&    0.4944&    0.2557&    0.2499\\
%   19  &  1  &  2   & 3    &0.5102 &   0.2303    &0.2595\\
%   20   &25 &  32    &4   & 0.5876  &  0.2149   & 0.1975\\
%   21&    7&   15&    2  &  0.3291   & 0.3323  &  0.3387\\
%   22 &  33   & 8 &   3 &   0.3959    &0.2039 &   0.4003\\
%   23  & 14  & 19  &  1&    0.6013&    0.1810&    0.2177\\
%   24   & 1 &   2   & 3    &0.3525 &   0.3283    &0.3192\\
%   25&   14&   30&   32   & 0.5784  &  0.2554   & 0.1662\\
%   26 &  16   &31 &  18  &  0.2366   & 0.4499  &  0.3134\\
%   27  & 32  & 12  &  5 &   0.4881    &0.3443 &   0.1676\\
%   28   & 1 &   2   & 3&    0.5154&    0.2345&    0.2501\\
%   29    &5&   31&   17    &0.4353 &   0.3402    &0.2245\\
%   30&   12   & 9 &  25   & 0.5005  &  0.2527   & 0.2468\\
%   31 &  26  & 13  & 29  &  0.3177   & 0.3379  &  0.3444\\
%   32  & 20 &  13   &30 &   0.3486    &0.3107 &   0.3406\\
%   33   & 1&    2    &3&    0.4003&    0.4392&    0.1605\\
   \hline
    \end{tabular}
    \label{Table1}
\end{table}

\begin{figure}
 \centering
 \subfigure[$\ddot{x}_{1,HT}^*(t)$]{\includegraphics[width=0.32\linewidth]{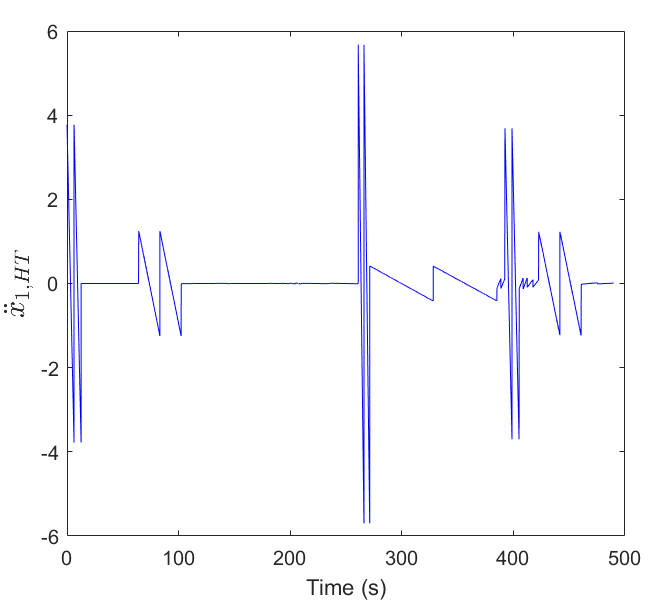}}
  \subfigure[$\ddot{x}_{2,HT}^*(t)$]{\includegraphics[width=0.32\linewidth]{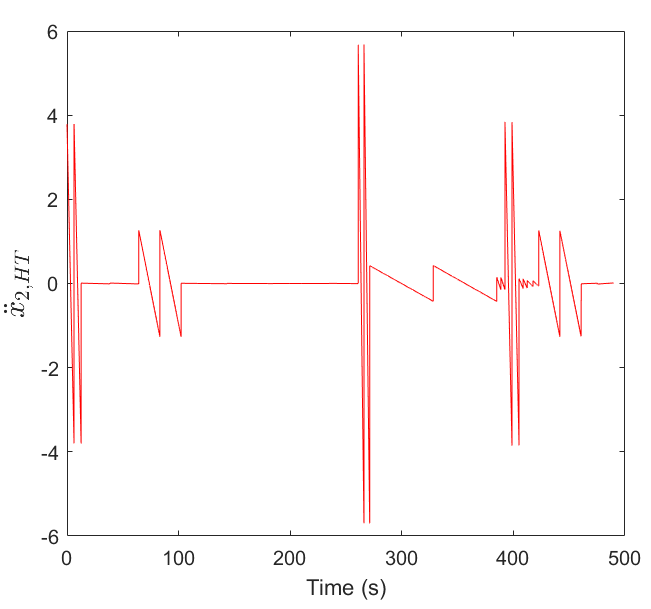}}
  \subfigure[$\ddot{x}_{3,HT}^*(t)$]{\includegraphics[width=0.32\linewidth]{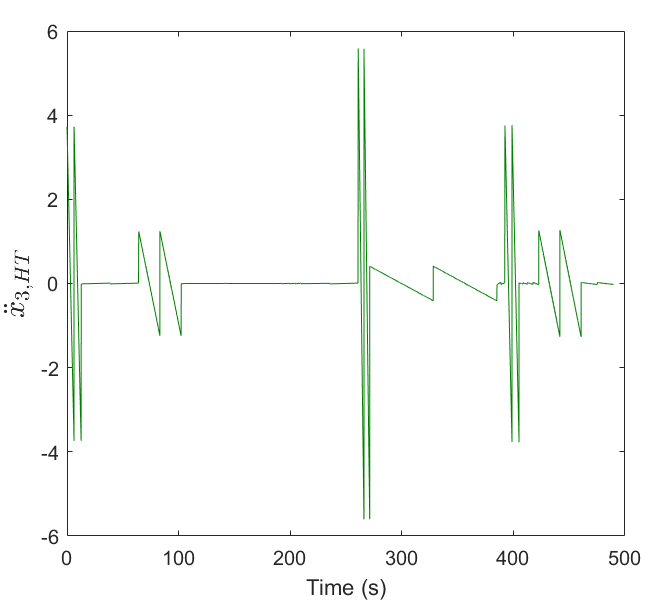}}
  \subfigure[$\ddot{y}_{1,HT}^*(t)$]{\includegraphics[width=0.32\linewidth]{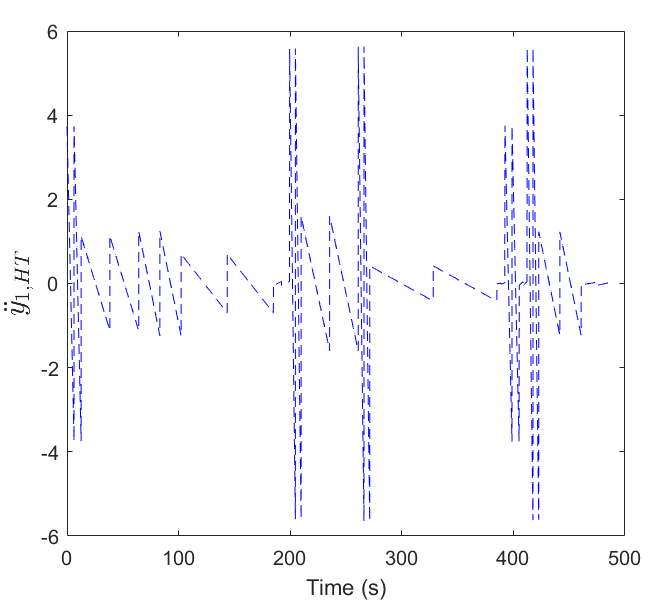}}
  \subfigure[$\ddot{y}_{2,HT}^*(t)$]{\includegraphics[width=0.32\linewidth]{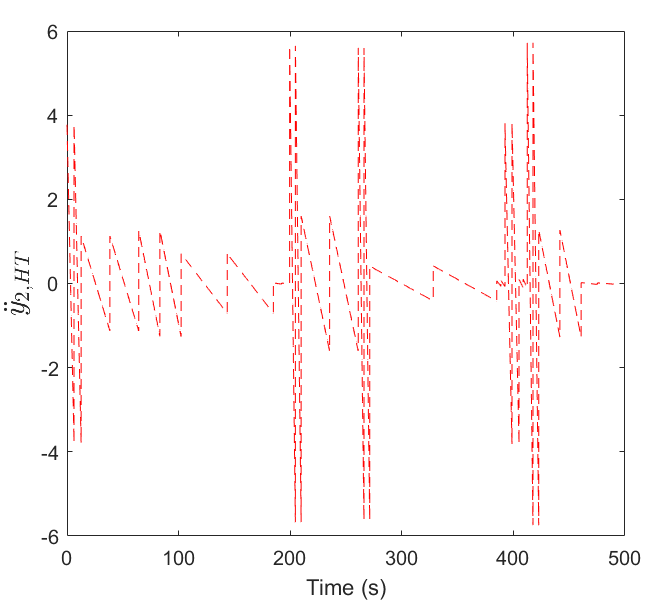}}
  \subfigure[$\ddot{y}_{3,HT}^*(t)$]{\includegraphics[width=0.32\linewidth]{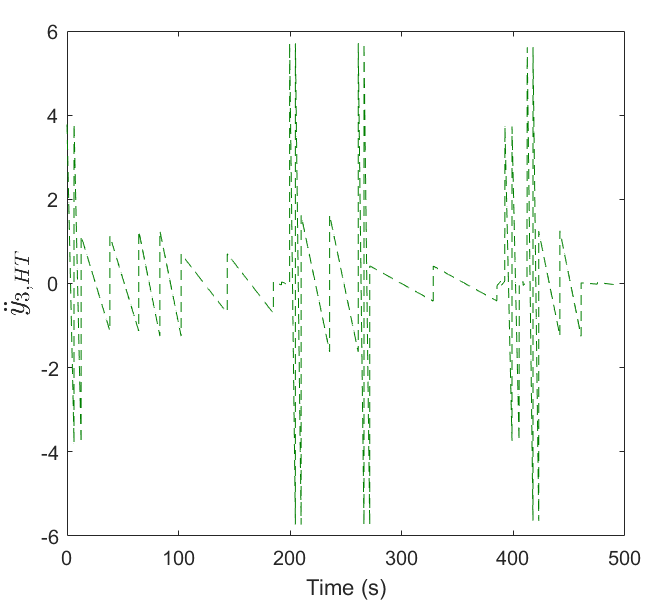}}
     \caption{Components of optimal control input $\mathbf{u}_L^*$ versus time for $t\in [0,490]s$.}
\label{accelerationha}
\end{figure}

\textbf{Safety Specification:} We assume that every quacopter can be enclosed by a ball of radius $\epsilon=0.45m$. 
For the initial formation shown in Fig. \ref{Formation} (a), $d_{\mathrm{min}}=3.5652m$ is the minimum separation distance between every two quadcopters. Furthermore, $\sigma_{\mathrm{min}}={1\over \sigma_{1,n_\tau}}=0.83$ is the lower bound for the eigenvalues of matrix $\mathbf{U}_{xy}$. Per Eq. \eqref{interagentcollision},
\[
\delta={1\over2}\left(d_{\mathrm{min}}\sigma_{\mathrm{min}}-2\epsilon\right)=1.04m
\]
is the  upper-bound for deviation of every quadcopter from its global desired position at any time $t\in \left[t_0,t_u\right]$.
% \begin{figure}
%  \centering
%  \subfigure[]{\includegraphics[width=0.5\linewidth]{eigenvalues.jpg}}
%   \subfigure[]{\includegraphics[width=0.45\linewidth]{beta3angle.jpg}}
%      \caption{(a) Eigenvalues $\lambda_1$ and $\lambda_2$ versus time $t$. (b) Rotation angle $\beta_3$ versus time $t$.}
% \label{eigenvalues}
% \end{figure}
\begin{figure}[htb]
\centering
\includegraphics[width=3.3  in]{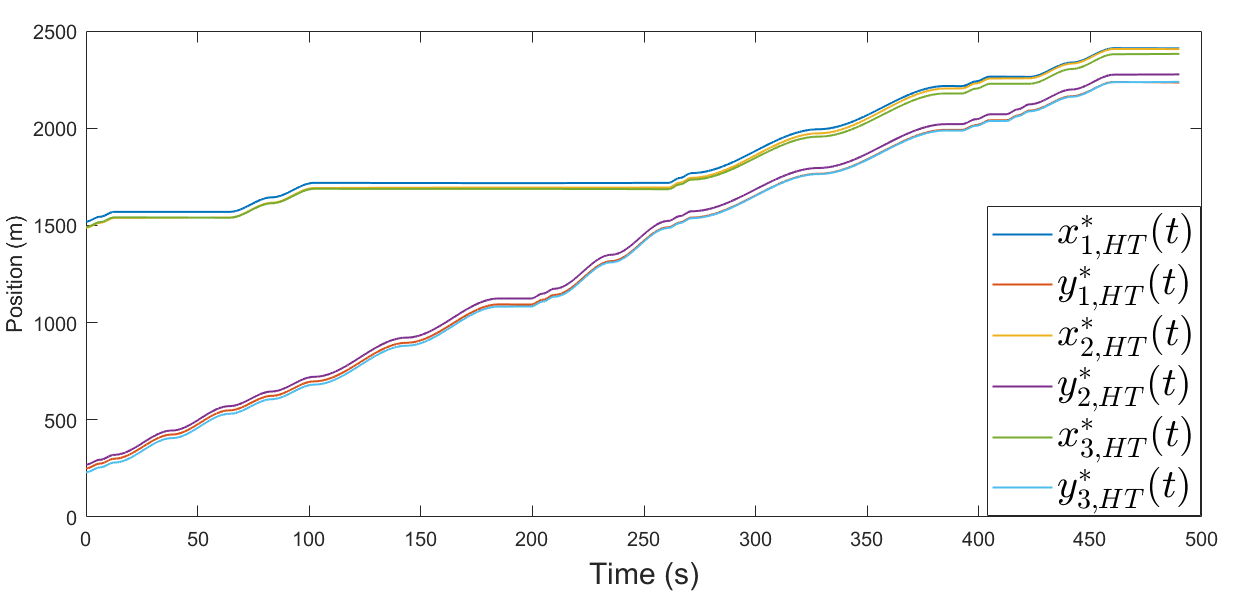}
\caption{Components of the optimal desired trajectories of the leaders for over time interval $\left[0,490\right]s$. }
\label{leadersdesired}
\end{figure}

\textbf{MQS Planning:}  It is desired that the MQS remains inside a ball of radius $r_{\mathrm{max}}=50m$ at any time $t\in [t_0,t_u]$.  By using A* search method,  the optimal intermediate  waypoints of the center of the containment ball are obtained. Then, the optimal path of the containment ball  is assigned and shown in Fig. \ref{MQS}. Given the intermediate waypoints of the center of containment ball, the desired trajectories of the leaders are determined by solving the constrained optimal control problem given in Section \ref{Planning3}. Given $t_s=0s$ and $\delta=1.04m$, $t_u=490s$ is assigned by using Algorithm \ref{euclid33}. Components of the optimal control input vector $\mathbf{u}_L^*(t)$, $\ddot{x}_{1,HT}^*(t)$, $\ddot{x}_{2,HT}^*(t)$, $\ddot{x}_{3,HT}^*(t)$, $\ddot{y}_{1,HT}^*(t)$, $\ddot{y}_{2,HT}^*(t)$, and $\ddot{y}_{3,HT}^*(t)$, and components of global desired positions of leaders, ${x}_{1,HT}^*(t)$, ${x}_{2,HT}^*(t)$, ${x}_{3,HT}^*(t)$, ${y}_{1,HT}^*(t)$, ${y}_{2,HT}^*(t)$, and ${y}_{3,HT}^*(t)$,  are plotted versus time $t$ in Figs. \ref{accelerationha} and \ref{leadersdesired}, respectively. Furthermore, deviation of every quadcopter from the global desired position is plotted in Fig. \ref{error-v5}. It is seen that deviation of no quadcopter exceeds $\delta=1.04m$ at any time $t\in [0,490]s$.

\begin{figure}[htb]
\centering
\includegraphics[width=3.3  in]{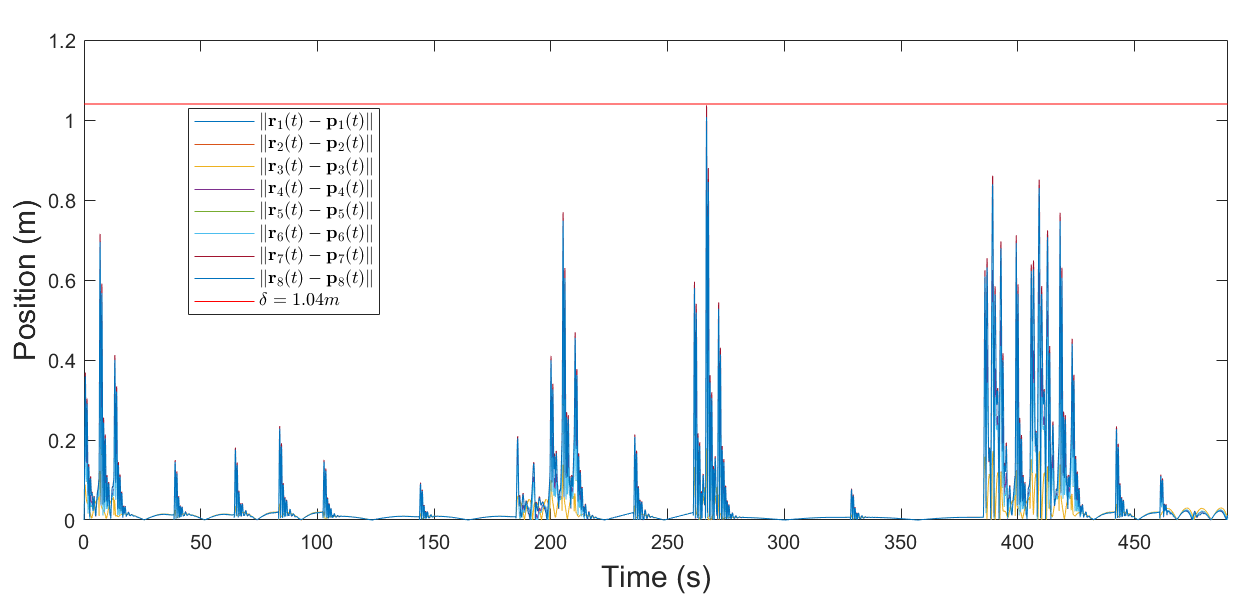}
\caption{Deviation of every quadcopter from its global desired trajectory  desired position over time interval $[0,490]s$. }
\label{error-v5}
\end{figure}

\section{Conclusion}\label{Conclusion}
This paper developed an algorithmic and formal approach for continuum deformation planning of a multi-quadcopter system coordinating in a geometrically-constrained environment. By using the principles of Lagrangian continuum mechanics, we obtained safety conditions for inter-agent collision avoidance and  follower containment through constraining the eigenvalues of the Jacobian matrix of the continuum deformation coordination.  To obtain safe and optimal transport of the MQS, we contain the MQS  by a rigid ball, and determine the intermediate waypoints of the containment ball using the A* search method. Given the intermediate configuration of the containment ball, we first determined the leaders' intermediate configurations by decomposing the homogeneous deformation coordination. Then, we assigned the optimal desired trajectories of the leader quadcopters by solving a constrained optimal control problem.

\section{\hspace{0.3cm}Acknowledgement}
This work has been supported by the National Science
Foundation under Award Nos. 1914581 and 1739525. The author gratefully thanks Professor Ella Atkins.

\bibliographystyle{IEEEtran}
\bibliography{reference}
\begin{IEEEbiography}[{\includegraphics[width=1in,height=1.25in,clip,keepaspectratio]{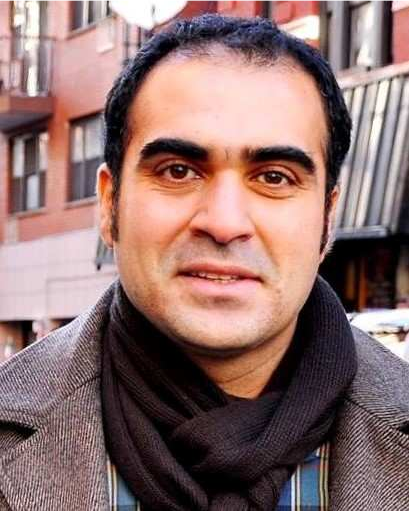}}]
{\textbf{Hossein Rastgoftar}} an Assistant Professor at Villanova University and an Adjunct Assistant Professor at the University of Michigan. He was an Assistant Research Scientist in the Aerospace Engineering Department from 2017 to 2020. Prior to that he was a postdoctoral researcher at the University of Michigan from 2015 to 2017. He received the B.Sc. degree in mechanical engineering-thermo-fluids from Shiraz University, Shiraz, Iran, the M.S. degrees in mechanical systems and solid mechanics from Shiraz University and the University of Central Florida, Orlando, FL, USA, and the Ph.D. degree in mechanical engineering from Drexel University, Philadelphia, in 2015. His current research interests include dynamics and control, multiagent systems, cyber-physical systems, and optimization and Markov decision processes.
\end{IEEEbiography}
\end{document}